\documentclass[12pt,psamsfonts]{amsart}

\usepackage{amssymb,amsfonts,amsmath}
\usepackage{enumerate}
\usepackage{mathrsfs}
\usepackage{fullpage}
\usepackage{xspace}
\usepackage[margin=1.0in]{geometry}
\usepackage{tcolorbox}
\usepackage{tikz-cd}
\usepackage{color}
\usepackage[all,arc]{xy}
\usepackage{aliascnt}
\usepackage[foot]{amsaddr}
\usepackage{hyperref}
\usepackage{graphicx}

\usepackage{enumitem}
\setlist[enumerate]{label=$(\mathrm{\arabic*})$, leftmargin=*}
\setlist[itemize]{leftmargin=*}

\newtheorem{thm}{Theorem}[section]

\newaliascnt{theo}{thm}
\newtheorem{theo}[theo]{Theorem}
\aliascntresetthe{theo}

\newaliascnt{cor}{thm}
\newtheorem{cor}[cor]{Corollary}
\aliascntresetthe{cor}

\newaliascnt{prop}{thm}
\newtheorem{prop}[prop]{Proposition}
\aliascntresetthe{prop}

\newaliascnt{lem}{thm}
\newtheorem{lem}[lem]{Lemma}
\aliascntresetthe{lem}

\newaliascnt{conj}{thm}
\newtheorem{conj}[conj]{Conjecture}
\aliascntresetthe{conj}

\newaliascnt{que}{thm}
\newtheorem{que}[que]{Question}
\aliascntresetthe{que}

\newaliascnt{ass}{thm}
\newtheorem{ass}[ass]{Assumption}
\aliascntresetthe{ass}

\newaliascnt{defnot}{thm}

\aliascntresetthe{defnot}

\theoremstyle{remark}
\newaliascnt{rem}{thm}
\newtheorem{rem}[rem]{Remark}
\aliascntresetthe{rem}

\theoremstyle{definition}
\newtheorem{defn}[thm]{Definition}

\newtheorem{exmp}[thm]{Example}

\newtheorem{notn}[thm]{Notation}

\newcommand{\Z}{\mathbb{Z}\xspace}

\newcommand{\F}{\mathbb{F}\xspace}
\newcommand{\Q}{\mathbb{Q}\xspace}
\newcommand{\G}{\mathbb{G}\xspace}
\DeclareMathOperator{\Spec}{Spec}
\DeclareMathOperator{\res}{res}
\DeclareMathOperator{\tor}{tor}

\DeclareMathOperator{\dv}{div}
\DeclareMathOperator{\alb}{alb}
\DeclareMathOperator{\img}{Im}

\DeclareMathOperator{\Cor}{Cor}

\DeclareMathOperator{\Hom}{Hom}
\DeclareMathOperator{\Pic}{Pic}
\DeclareMathOperator{\Gal}{Gal}

\DeclareMathOperator{\Br}{Br}
\DeclareMathOperator{\inv}{inv}
\DeclareMathOperator{\nd}{nd}
\DeclareMathOperator{\CH}{CH}
\DeclareMathOperator{\rk}{rank}
\DeclareMathOperator{\End}{End}
\DeclareMathOperator{\NS}{NS}
\DeclareMathOperator{\pr}{pr}
\makeatletter
\let\c@equation\c@thm
\makeatother
\numberwithin{equation}{section}

\newcommand{\blue}{\color{blue}}
\newcommand{\black}{\color{black}}

\newcommand{\computationsUpperBound}{5000}
\newcommand{\missingCurves}{176}

\title{Weak Approximation for $0$-cycles on a product of elliptic curves}

\author[*]{Evangelia Gazaki* \\
\\with an appendix by Angelos Koutsianas**} \address[*]{\normalfont Department of Mathematics, University of Virginia, 221 Kerchof Hall, 141 Cabell Dr., Charlottesville, VA, 22904, USA. Email: \texttt{eg4va@virginia.edu}}
\address[**]{\normalfont Department of Mathematics, Aristotle University of Thessaloniki, 54124, Thessaloniki, Greece. Email: \texttt{akoutsianas@math.auth.gr}}
\begin{document}

\maketitle

\begin{abstract} In the 1980's Colliot-Th\'{e}l\`{e}ne, Sansuc, Kato and S. Saito proposed conjectures related to local-to-global principles for $0$-cycles on arbitrary smooth projective varieties over a number field. We give some evidence for these conjectures for a product $X=E_1\times E_2$ of two elliptic curves. In the special case when $X=E\times E$ is the self-product of an elliptic curve $E$ over $\Q$ with potential complex multiplication, we show that the places of good ordinary reduction are often involved in a Brauer-Manin obstruction for $0$-cycles over a finite base change. We give many examples when these $0$-cycles can be lifted to global ones. 
\end{abstract}

\section{Introduction} 
Let $X$ be a smooth projective geometrically connected variety over a number field $F$. We denote by $\Omega$ the set of all places $v$ of $F$ and by $F_v$ the completion of $F$ at a place $v$. The classical local-to-global principles for $X$ refer to the image of the diagonal embedding \[X(F)\hookrightarrow X(\mathbf{A}_F):=\prod_{v\in\Omega}X(F_v)\] to the set $X(\mathbf{A}_F)$ of adelic points of $X$. The cohomological Brauer group $\Br(X):=H^2(X_{\text{\'{e}t}},\G_m)$ is known to often obstruct either the existence of an $F$-rational point or the density of $X(F)$ in $X(\mathbf{A}_F)$. Namely, by the foundational work of Y. Manin (\cite{Manin1971}), the Brauer group gives rise to an intermediate closed subset $X(F)\subseteq X(\mathbf{A}_F)^{\Br(X)}\subseteq X(\mathbf{A}_F)$, which is often empty or properly contained in $X(\mathbf{A}_F)$.
   Unfortunately, the Brauer-Manin obstruction cannot always explain the failure of the Hasse principle and weak approximation for points (cf. \cite{Skorobogatov1999, Poonen2010}). However, this happens to be the case for abelian varieties, assuming finiteness of their Tate-Shafarevich groups, and it is conjectured to be the case for geometrically rationally connected varieties (\cite[p.~174]{Colliot-Thelene2003}). The answer is likely to be yes also for $K3$-surfaces (cf. \cite[p.~4]{Skorobogatov/Zharin2008}). 
   
  
  In this article we are interested in an analog of this study for $0$-cycles, namely for the group $\CH_0(X)$.
   Denote by $\Omega_f(F)$ (resp. $\Omega_\infty(F)$) the set of all finite (resp. infinite) places of $F$. For a place $v\in\Omega$ denote by $X_v$ the base change to $F_v$.  We have the following conjecture.
  \begin{conj}[{\cite[Section 4]{Colliot-Thelene/Sansuc1981}, \cite[Section 7]{Kato/Saito1986}, see also \cite[Conjecture~1.5 (c)]{Colliot-Thelene1993}}]
\label{locatoglobalconj} 
Let $X$ be a smooth projective geometrically connected variety over a  number field $F$. 
 The following complex is exact, 
\[\hspace{10pt}\varprojlim_{n}  \CH_0(X)/n\rightarrow
\varprojlim_{n}\CH_{0,\mathbf{A}}(X)/n\rightarrow\Hom(\Br(X),\Q/\Z).\] 
\end{conj}
 The above formulation is due to van Hamel (\cite{vanHamel2003}), see also \cite[Conjecture~($E_0$)]{Wittenberg2012}.  The adelic Chow group $\CH_{0,\mathbf{A}}(X)$ is equal to $\displaystyle\prod_{v\in\Omega_f(F)}\CH_0(X_v)$ when $F$ is totally imaginary, and it has a small contribution from the infinite real places otherwise.
 
When $X$ is a smooth projective curve, \autoref{locatoglobalconj} has been proved by Colliot-Th\'{e}l\`{e}ne (\cite[paragraph 3]{CT1997}) assuming the finiteness of the Tate-Shafarevich group of its Jacobian. 
In higher dimensions, the main evidence for \autoref{locatoglobalconj} is for rationally connected varieties starting with the work of Colliot-Th\'{e}l\`{e}ne, Sansuc and Swinnerton-Dyer (\cite{CT/Sansuc/SwinnertonDyerI, CT/Sansuc/SwinnertonDyerII}) and continued by the work of Liang (\cite[Theorem B]{Liang2013}). The latter proved this conjecture assuming that the Brauer-Manin obstruction is the only obstruction to Weak Approximation for points on $X_L$ where $L$ is any finite extension of $F$. There is some recent partial evidence of Ieronymou of similar flavor for $K3$-surfaces (cf. \cite[Theorem 1.2]{Ieronymou2021}). Moreover, Harpaz and Wittenberg (\cite[Theorem 1.3]{Wittenberg/Harpaz2016})  proved that the conjecture is compatible with fibrations over curves. 

The purpose of this article is to give some evidence for this conjecture for a product $X=E_1\times E_2$ of elliptic curves. We focus on the following weaker question. 
\begin{que} Let $p$ be a prime number and $\Br(X)\{p\}$ be the $p$-primary torsion subgroup of $\Br(X)$. Is the following complex exact
\begin{equation}\label{complexintro}\varprojlim_{n}  \CH_0(X)/p^n\rightarrow
\varprojlim_{n}\CH_{0,\mathbf{A}}(X)/p^n\rightarrow\Hom(\Br(X)\{p\},\Q/\Z)?\end{equation} 
\end{que} Assuming the finiteness of the Tate-Shafarevich group of $X$, and that the action of the absolute Galois group $G_F$ on the N\'{e}ron-Severi group $\NS(X\otimes_F\overline{F})$ is trivial, the exactness of \ref{complexintro} can be reduced (cf. \autoref{compatibility}, \autoref{reduction}) to the exactness of the complex,
\begin{equation}\label{complex2intro}\varprojlim_{n}  F^2(X)/p^n\rightarrow
\varprojlim_{n} F^2_\mathbf{A}(X)/p^n\rightarrow\Hom\left(\frac{\Br(X)\{p\}}{\Br_1(X)\{p\}},\Q/\Z\right),\end{equation}
where $F^2(X)$ is the kernel of the Albanese map of $X$ and $\Br_1(X)$ is the algebraic Brauer group of $X$. From now make the above assumptions. 
Our first result is the following theorem. 
\begin{theo}\label{main0} (\autoref{mainmain0}) Let $X=E_1\times E_2$ be a product of elliptic curves over a number field $F$. There is an infinite set $T$ of primes $p$ for which the middle term in \ref{complex2intro} vanishes, making the complex exact. If we further assume that at least one of the curves does not have potential complex multiplication, then the complement $S$ of $T$ is a set of primes of density zero (in the sense of \cite{Serre1981}). 
\end{theo} 
 When both $E_1, E_2$ have potentially good reduction, this result was already obtained in \cite[Corollary 1.8]{Gazaki/Hiranouchi2021}.
 It follows by \cite[Theorem 3.5]{Raskind/Spiess2000} that the only places that might contribute a nontrivial factor in $\varprojlim_{n} F^2_\mathbf{A}(X)/p^n$ are the places of bad reduction and the places above $p$. The key to prove  \autoref{main0} is that under some assumptions on the prime $p$, the group $F^2(X_v)$ is $p$-divisible for every unramified place $v$ above $p$ of good reduction (cf.~\cite[Theorem 1.1]{Gazaki/Hiranouchi2021}). 
  One can think of this result as a weaker analog of the vanishing $F^2(X_v)=0$ for $X$ a rationally connected variety and $v$ a place of good reduction (\cite[Theorem 5]{Kollar/Szabo2003}).
  
   A natural follow-up question is whether the places of good reduction are ever involved in a Brauer-Manin obstruction for $0$-cycles. Our second theorem gives an affirmative answer for places of good ordinary reduction that are ramified enough. 
 
\begin{theo}\label{main1} (cf. \autoref{mainmain1}) Let $X=E\times E$ be the self product of an elliptic curve $E$ over $\Q$. Suppose that $E\otimes_\Q\overline{\Q}$ has complex multiplication by the full ring of integers of a quadratic imaginary field $K$. Let $p\geq 5$ be a prime which is coprime to the conductor $\mathfrak{n}$ of $E$ and suppose that $p$ splits completely in $K$, $p=\pi\overline{\pi}$ for some prime element $\pi$ of $\mathcal{O}_K$. Let $\overline{E_p}$ be the reduction of $E$ modulo $p$.  Suppose there exists a finite Galois extension  $F_0/K$  of degree $n<p-1$ such that there is a unique unramified place $w$ of $F_0$ above $\pi$ with the property that $p$ divides $|\overline{E_p}(\F_w)|$, where $\F_w$ is the residue field of $F_0$ at $w$.
  Then there exists a finite Galois extension $L/F_0$ of degree $p-1$, totally ramified at $w$ such that the following are true for the group $\displaystyle\varprojlim\limits_n F^2_{\mathbf{A}}(X_L)/p^n$. 
\begin{enumerate}
\item It  is equal to $\displaystyle\prod_{v|p}\varprojlim\limits_n F^2(X_{L_v})/p^n$, and isomorphic to $\Z/p\oplus\Z/p$. 
\item It is orthogonal to the transcendental Brauer group $\Br(X_L)/\Br_1(X_L)$. 
\end{enumerate}  
  
The same result holds if the elliptic curve $E$ is defined over $K$. 
\end{theo} 

The primes $p$ of good reduction that split completely in $K$ are precisely the primes of good \textbf{ordinary reduction}. 
The involvement of the ordinary reduction places in a Brauer-Manin obstruction after suitable base change is not surprising. An analogous result for rational points was recently obtained in \cite{BrightNewton2020}, which in particular applies to weak approximation for points on abelian varieties and $K3$ surfaces. See also \cite{Pagano2021} for an explicit example. 

\subsection*{Global Approximation} We next pass to the main question of this article namely, whether in the situation considered in \autoref{main1}, one can lift the group $\varprojlim\limits_n F^2_{\mathbf{A}}(X_{L})/p^n$ to global elements. An important observation is that $\varprojlim\limits_n F^2_{\mathbf{A}}(X_{L})/p^n$ coincides with $\displaystyle\prod_{v|p}F^2(X_{L_v})/p$ (cf. \autoref{localAlb}). Thus, one is looking for genuine $0$-cycles and the question is reduced to showing that the diagonal map 
$F^2(X_L)/p\xrightarrow{\Delta}\prod_{v|p}F^2(X_{L_v})/p$ is surjective. 

We focus mainly on the simplest case considered in \autoref{main1}, namely when $F_0=\Q$; that is, when the reduction of the elliptic curve $E$ satisfies $|\overline{E_p}(\F_p)|=p$. 
One can find infinite families with this property. For example, consider the family of elliptic curves with Weierstrass equation $\{y^2=x^3+c;\;c\in\Z, c\neq 0\}$ having potential CM by the ring of integers of $\Q(\zeta_3)$. Let $p$ be a prime of the form $4p=1+3v^2$. Then, there exist exactly $\frac{p-1}{6}$ congruence classes $c\mod p$ that give elliptic curves  that satisfy $|\overline{E_p}(\F_p)|=p$.


In \autoref{computations_section} we give various sufficient conditions that guarantee the ability to lift. The conditions are supported by explicit examples, giving therefore some evidence for \autoref{locatoglobalconj}. The following \autoref{main2} is indicative of the computations we do in \autoref{CMsection}.
 Before stating the result, we need to refer to some notation. 

Suppose the pair $E,p$ satisfies the assumptions of \autoref{main1}. 
 Let $P\in E(\Q)$. We will denote by $P_{\Q_p}$ the image of $P$ under the restriction map $E(\Q)\hookrightarrow E(\Q_p)$. We abuse notation and denote by $P_{\Q_p}$ also its image modulo $pE(\Q_p)$. In \autoref{decompose} we construct a decomposition $P_{\Q_p}=\widehat{P}_{\Q_p}\oplus \overline{P}_{\Q_p}$, with $\widehat{P}_{\Q_p}\in\widehat{E}(p\Z_p)/[p]$, and $\overline{P}_{\Q_p}\in \overline{E_p}(\F_p)/p$ (cf. notation \eqref{localpointsplit}). 
  We can now state our next theorem.

\begin{theo}\label{main2} Let $p\geq 5$ be a prime and $E$ an elliptic curve over $\Q$ satisfying the assumptions of \autoref{main1}. Assume further that $|\overline{E_p}(\F_p)|=p$. Let $L/K$ be the extension constructed in \autoref{main1}. Suppose that the Mordell-Weil group $E(\Q)$ has positive rank and there exists a point $P\in E(\Q)$ of infinite order such that the image $P_{\Q_p}$ of $P$ under the restriction map $\displaystyle\res_{\Q_p/\Q}:E(\Q)/p\rightarrow E(\Q_p)/p$, has the property that $\displaystyle\widehat{P}_{\Q_p}\neq 0\in \frac{\widehat{E}(p\Z_p)}{[p]\widehat{E}(p\Z_p)}$.  
Then we have a surjection \[F^2(X_L)/p\to\prod_{v|p}F^2(X_{L_v})/p\to 0.\] In particular, the complex \eqref{complex2intro} is exact. 

\end{theo} 

 We expect that \autoref{main2} is satisfied by infinite families of elliptic curves. The criterion given in \autoref{main2} can be checked computationally in SAGE for small values of the prime $p$. In the appendix \ref{appendix} we study the family of elliptic curves $\{y^2=x^3-2+7n,n\in\Z\}$ and the prime $p=7$. When $n$ lies in the interval $[-\computationsUpperBound, \computationsUpperBound]$, we show that about ~86.68\% of elliptic curves with rank one over $\Q$ have such a ``good point" $P\in E(\Q)$.
 
  The method used to prove \autoref{main2} can be generalized to more situations (cf. \autoref{quadraticcase} for when $F_0/K$ is a quadratic extension).  In \autoref{nogoodpoint} we give criteria of similar flavor for elliptic curves that do not possess a suitable rational point  of infinite order. For the purposes of the introduction, we chose to state our theorem in its simplest form. 
  
  \begin{rem} In a recent article Liang (\cite{Liang2020arxiv}) showed some compatibility of weak approximation for $0$-cycles with certain products using the fibration method. The same method has been used in most known results (\cite{Liang2013, Wittenberg/Harpaz2016, Ieronymou2021}). Unfortunately, this breakthrough method does not seem to generalize to abelian varieties. For, the key cohomological properties $H^i(X,\mathcal{O}_X)=0$, for $i=1,2$, and finiteness of $\Br(X)/\Br_0(X)$ (used for rationally connected varieties and $K3$ surfaces respectively) are no longer true. Our approach is different in the sense that it does not use the assumption that the weak approximation is the only one for rational points. The downside is that we cannot find a uniform treatment that guarantees lifting, but we instead present many different sufficient conditions. The advantages on the other hand are first that our results are unconditional and second that we construct global $0$-cycles on the nose. In fact, if one only focuses on proving exactness of \ref{complex2intro}, then the finiteness of the Tate-Shafarevich group is no longer needed. 
  \end{rem}
%
\subsection{Notation} Throughout this note unless otherwise mentioned we will be using the following notation.
\begin{itemize}
\item For a number field $F$, $\mathcal{O}_F$, $\Omega(F), \Omega_f(F), \Omega_\infty(F)$ will be respectively its ring of integers, the set of all places, all finite and all infinite places of $F$.
\item For a finite extension $k/\Q_p$, $\mathcal{O}_k$ will be its ring of integers, $\mathfrak{m}_k$ its maximal ideal and $\F_k$ its residue field. 
\item For a place $v\in\Omega(F)$, $F_v$ will be the completion of $F$ at $v$. 
\item For an abelian group $M$ and a positive integer $n$, $M_n$ and $M/n$ will be the $n$-torsion and $n$-cotorsion of $M$ respectively. Moreover, $\widehat{M}$ will be the completion, $\widehat{M}=\varprojlim\limits_{n}M/n$.
\item For an elliptic curve $E$ and an integer $n$ we will instead write $E[n]$ for the $n$-torsion. 
\item For a field extension $L/K$, and a variety $X$ over $K$, $X_L$ will be the base change to $L$. 
\item For a field $k$ and a continuous $\Gal(\overline{k}/k)$-module $M$ we will denote by $\{H^i(k,M)\}_{i\geq 0}$ the Galois cohomology groups of $M$. 
\end{itemize} 
\vspace{1pt}
\subsection{Acknowledgements} We are truly grateful to Professor Andrew Sutherland for computing for us many examples that satisfy the assumptions of \autoref{main1}. We heartily thank Professors Toshiro Hiranouchi, and Olivier Wittenberg for useful discussions. Additionally, we would like to express our deep gratitude to Professor Evis Ieronymou, who found a mistake in an earlier version of this paper. Finally, we would like to thank the referee for correcting our mistakes and suggesting many useful comments.  The first author was partially supported by the NSF grant DMS-2001605.

\vspace{2pt}

\section{Background}\label{sec:background}

\subsection{Elliptic Curves with Complex Multiplication}\label{CM background}
In this subsection we recall some necessary facts about elliptic curves with complex multiplication. Let $E$ be an elliptic curve over $\Q$ with potential complex multiplication by the full ring of integers $\mathcal{O}_K$ of $K$. It follows by \cite[Corollary 5.12]{Rubin1999} that $K$ has class number one.
Note that there exactly $9$ such fields $K=\Q(\sqrt{D})$. Namely,  $D\in\{-1,-2,-3,-7,-11,-19,-43,-67,-163\}$. The same result holds if $E$ is defined over $K$. 
\begin{notn}
For an element $a\in\mathcal{O}_K$ we will denote by $E[a]$ the kernel of the isogeny $a:E_K\rightarrow E_K$. Similarly, if $\mathfrak{a}$ is an ideal of $\mathcal{O}_K$, then $\mathfrak{a}=(a)$ for some element $a\in\mathcal{O}_K$ which is uniquely defined up to units. We will also write $E[\mathfrak{a}]$ for the kernel of $a:E_K\rightarrow E_K$ (which is independent of the choice of generator). 
\end{notn}

Let $\mathfrak{a}$ be a nonzero proper ideal of $\mathcal{O}_K$ which is prime to the conductor $\mathfrak{n}_K$ of $E_K$. It follows by \cite[Cor. 5.20 (ii)]{Rubin1999} that there is an isomorphism:
\[\Gal(K(E[\mathfrak{a}])/K)\xrightarrow{\simeq}(\mathcal{O}_K/\mathfrak{a})^\times.\] In particular, if $\mathfrak{p}$ is a prime of $\mathcal{O}_K$ not dividing the conductor, then the Galois group  $\Gal(K(E[\mathfrak{p}])/K)$ is cyclic of order $N_{K/\Q}(\mathfrak{p})-1$, where $N_{K/\Q}:K^\times\rightarrow\Q^\times$ is the norm map. 
Next suppose that $p$ is a rational prime coprime to $\mathfrak{n}$. We distinguish the following two cases:
 \begin{itemize}
 \item $p$ is a prime element of $\mathcal{O}_K$, so $k=K_{(p)}$  is a quadratic unramified extension of $\Q_p$. Then $\Gal(K(E[p])/K)\xrightarrow{\simeq}\F_{p^2}^\times$. In this case the elliptic curve $E_k$ has good supersingular reduction. 
 \item $p$ splits completely in $K$, that is, $p=\pi\overline{\pi}$, where $\pi$ is a prime element of $\mathcal{O}_K$. Then we have an isomorphism,
 $\Gal(K(E[p])/K)\xrightarrow{\simeq}\F_{p}^\times\oplus\F_{p}^\times,$ and hence $K(E[p])/K$ is an extension of degree $(p-1)^2$. In this case the elliptic curve $E_{\Q_p}$ has good ordinary reduction. 
 \end{itemize}
 The case of interest to us in this note is the latter. From now on we will refer to such primes as \textit{ordinary primes}. We fix such an ordinary prime $p$ and a factorization $p=\pi\overline{\pi}$ into prime elements of $\mathcal{O}_K$. We will denote by $\mathfrak{p}, \overline{\mathfrak{p}}$ the prime ideals $(\pi), (\overline{\pi})$ of $\mathcal{O}_K$ respectively. Notice that the two completions $K_\mathfrak{p}, K_{\overline{\mathfrak{p}}}$ of $K$ are both equal to $\Q_p$. Let $\overline{E_p}$ be the reduction of $E$ modulo $p$ and $r:E(\Q_p)\rightarrow\overline{E_p}(\F_p)$ be the reduction map. We will the use same notation for the reduction $E(k)\xrightarrow{r}\overline{E_p}(\F_k)$ where $k/\Q_p$ is any finite extension. Moreover, we will denote by $\widehat{E_{\Q_p}}$ the formal group of $E_{\Q_p}$.
 
  It follows  
by \cite{Deuring1941} (see also \cite[13.4, Theorem 12]{Lang1987}) that
we can choose $\pi$ so that the endomorphism $\pi:E_{\Q_p}\rightarrow E_{\Q_p}$ when reduced modulo $p$ coincides with the Frobenius endomorphism $\phi_p:\overline{E_p}\rightarrow\overline{E_p}$. In particular, the reduction of $\pi$ is an automorphism of $\overline{E_p}$, and hence has trivial kernel. This implies that for every $n\geq 1$ the subgroup $E[\pi^n]$ of $E[p^n]$ coincides with the $p^n$-torsion of the formal group, $\widehat{E_{\Q_p}}[p^n]$. Moreover, the relation $p=\pi\overline{\pi}$ implies that $\overline{\pi}$ induces a height zero isogeny $[\overline{\pi}]:\widehat{E_{\Q_p}}\to\widehat{E_{\Q_p}}$, and hence for every $n\geq 1$ the reduction induces an isomorphism of abelian groups $E[\overline{\pi}^n]\xrightarrow{\simeq}\overline{E_p}[p^n]$. 

Let $k/\Q_p$ be a finite extension. Since $E_{\Q_p}$ has good ordinary reduction, for every $n\geq 1$ there is a short exact sequence of $\Gal(\overline{k}/k)$-modules,
\begin{equation}\label{ses2}
0\rightarrow \widehat{E_{\Q_p}}[p^n]\rightarrow E[p^n]\rightarrow\overline{E_p}[p^n]\rightarrow 0. 
\end{equation} The above discussion shows that \ref{ses2} splits, since the subgroup $E[\overline{\pi}^n]$ of $E[p^n]$ maps isomorphically to $\overline{E_p}[p^n]$ (see also \cite[A.2.4]{Serre89} for more general results). 

 
 \subsection*{Special Fiber} We next recall formulas for $|\overline{E_p}(\F_p)|$ for the various choices of the quadratic imaginary field $K$. We are particularly interested in finding examples of curves that satisfy $|\overline{E_p}(\F_p)|=p$. That is, curves for which the extension $F_0$ of \autoref{main1} can be taken to be $\Q$. In all cases we have a formula,
 \begin{equation}
 |\overline{E_p}(\F_p)|=p+1-(\pi+\overline{\pi}),
\end{equation}  where $\pi\overline{\pi}=p$ and $\pi$ reduces mod $p$ to the Frobenius (cf. \cite[p. 1]{Joux/Morain1995}). For the various choices of quadratic imaginary field $K$, the correct choice of prime element $\pi$ can be computed. We list a few examples. 

\begin{exmp}\label{ex1} Suppose $D=-3$, that is, $K=\Q(\zeta_3)$, and $E$ is given by the Weierstrass equation $y^2=x^3+c$ with $c\in\Z$. It follows by \cite[Theorem 1]{Rajwade1969} that 
 \begin{equation}\label{formula1}
 |\overline{E_p}(\F_p)|=p+1-\left(\frac{4c}{\pi_0}\right)_6\cdot\overline{\pi}_0-\left(\frac{4c}{\overline{\pi}_0}\right)_6\cdot\pi_0,
 \end{equation} where $\pi_0$ is a prime element of $K$ such that $\pi_0\overline{\pi}_0=p$ and $\pi_0,\overline{\pi}_0$ are normalized, that is, they are congruent to $1\mod 3$. Here $\displaystyle\left(\frac{a}{\pi_0}\right)_6=a^{\frac{p-1}{6}}(\text{mod}\;\pi_0)$ is the sixth power residue symbol. The symbol $\displaystyle\left(\frac{4c}{\pi_0}\right)_6$ can take any value within the set of units $\{\pm 1, \pm\zeta_3,\pm\zeta_3^2\}$ of $\Z[\zeta_3]$. In particular, when $p$ is of the form $4p=1+3v^2$ for some $v\in\Z$, there are exactly $\frac{p-1}{6}$ different reductions $\overline{E_p}$ which satisfy $|\overline{E_p}(\F_p)|=p$. Primes of this form include $p=7, 37, 61,\ldots$. For example, when $p=7$, the family of elliptic curves $\{E_n: y^2=x^3-2+7n,n\in\Z\}$ satisfies the desired equality. Later in this article we will make this family a case study to construct examples that satisfy \autoref{main2}. A second example is given for $p=61$ and the family $\{E_t:y^2=x^3+2+61t,t\in\Z\}$. 
\end{exmp}
\begin{exmp} Suppose $D=-11$. A sage computation shows that for $p=223$ the family $\{E_s: y^2=x^3-1056x+13552+223s,s\in\Z\}$ satisfies the desired equality. Another class of examples is given for $D=-19$, $p=43$ and $\{E_l:y^2=x^3-152x+722+43l,l\in\Z\}$. 
\end{exmp}
\begin{exmp} When $D=-43, -67, -163$, and $E$ is given by a CM Weierstrass equation  with parameter $c$ (cf. \cite[Tableau 1]{Joux/Morain1995}) it follows by \cite[Th\'{e}or\`{e}me 1]{Joux/Morain1995} that 
 \begin{equation}\label{formula2}
 |\overline{E_p}(\F_p)|=p+1-\left(\frac{2}{p}\right)\left(\frac{u}{p}\right)\left(\frac{c}{p}\right)u,
 \end{equation} where $\left(\frac{a}{b}\right)$ is the Legendre symbol, and the integer $u$ is such that $4p=u^2-Dv^2$. In this case it is easy to see that if $u=1$, then every integer $c$ such that $\displaystyle\left(\frac{2}{p}\right)\left(\frac{c}{p}\right)=-1$ gives $|\overline{E_p}(\F_p)|=p$. For example for $D=-43$ the following primes are of the form $1+43v^2$: $p=11, 97, 269, 1301,\ldots$.   
\end{exmp}
\begin{exmp}\label{Zi} When $D=-1$ or $-2$, it follows that $\pi+\overline{\pi}$ is always an even integer, and hence the equality $|\overline{E_p}(\F_p)|=p$ never happens. That is, any extension $F_0$ satisfying the assumptions of \autoref{main1} has positive degree. 
 For example, consider the family of elliptic curves given by the Weierstrass equation $y^2=x^3+(3+5n)x$, with $n\in\Z$, which has potential complex multiplication by $\Z[i]$. Consider the prime $p=5$, which splits completely in $\Q(i)$. A SAGE computation shows that if $E$ is any elliptic curve in the family, then $p||\overline{E_p}(\F_{p^2})|$. 
\end{exmp}

\vspace{1pt}
\subsection{$0$-cycles and Somekawa $K$-groups}\label{Somekawa} 
Let $X$ be a smooth projective variety over a perfect field $k$.  We consider the Chow group of $0$-cycles, $\CH_0(X)$. We recall that this group has a filtration \[\CH_0(X)\supset F^1(X)\supset F^2(X)\supset 0,\] where $F^1(X):=\ker(\deg: \CH_0(X) \to \Z)$ is the kernel of the degree map, and $F^2(X):=\ker(\alb_X:F^1(X)\to \mathrm{Alb}_X(k))$ is the kernel of the Albanese map. When $X=C_1\times C_2$ is a product of two smooth projective, geometrically  connected curves over $k$ such that $X(k)\neq\emptyset$, Raskind and Spiess (\cite[Theorem 2.2, Corollary 2.4.1]{Raskind/Spiess2000}) showed an isomorphism 
\begin{equation}\label{Kiso}
F^2(X)\simeq K(k;J_1,J_2),
\end{equation} where $K(k;J_1,J_2)$ is  the Somekawa $K$-group attached to the Jacobian varieties $J_1,J_2$ of $C_1, C_2$. The group $K(k;J_1,J_2)$, defined by K. Kato and Somekawa in \cite{Somekawa1990}, is a quotient of $\displaystyle\bigoplus_{L/k\text{ finite}}J_1(L)\otimes J_2(L)$ by two relations. The first relation is known as \textit{projection formula} and the second as \textit{Weil reciprocity}. In this article we won't make explicit use of these relations, and hence we omit the precise definition (see \cite[Definition 2.1.1]{Raskind/Spiess2000} and the footnote on p. 10 for a correction to Somekawa's original definition).
\begin{exmp}\label{selfproduct}
Suppose that $X=E\times E$ is the self-product of an elliptic curve $E$ over $k$. We denote by $K_2(k;E)$ the group $K(k;E,E)$. Moreover we will denote by $O$ the zero element of $E$. The Albanese kernel $F^2(X)$ is generated by $0$-cycles of the form 
\[w_{P,Q}:=f_{L/k\star}([P,Q]-[P,O]-[Q,O]+[O,O]),\] where $L/k$ runs through all finite extensions of $k$, $P,Q\in E(L)$, and $f_{L/k}$ is the proper push-forward $\CH_0(X_L)\xrightarrow{f_{L/k}}\CH_0(X)$. 
 The isomorphism \eqref{Kiso} sends the $0$-cycle $w_{P,Q}$ to the symbol 
$\{P,Q\}_{L/k}$. From now on we identify these two objects. 
\end{exmp}

\vspace{1pt}
\subsection{Weak Approximation for $0$-cycles} Let $X$ be a smooth projective geometrically connected variety over a number field $F$. Let $v\in\Omega_f(F)$. There is a local pairing, 
\[\langle\cdot,\cdot\rangle_v:\CH_0(X_v)\times\Br(X_v)\rightarrow\Br(F_v)\simeq\Q/\Z\]
called the \textit{Brauer-Manin pairing} defined as follows. For a closed point $P\in X_v$ and a Brauer class $\alpha \in \Br(X_v)$, 
the pull-back of $\alpha$ along $P$ 
is denoted by $\alpha(P) \in \Br(F_v(P))$, where $F_v(P)$ is the residue field of $P$. 
The pairing $\langle\cdot,\cdot\rangle_v$ is defined on points by $\langle P,\alpha\rangle_v := \mathrm{Cor}_{F_v(P)/F_v}(\alpha(P))$ and it factors through rational equivalence (cf. \cite[p.~4]{Colliot-Thelene1993}). Here $\Cor_{F_v(P)/F_v}$ is the Corestriction map of Galois cohomology (cf. \cite[VIII.2]{Serre1979local}). 

\begin{defn} Suppose that $F$ is totally imaginary. The adelic Chow group $\CH_{0,\mathbf{A}}(X)$ is defined to be the product $\displaystyle \CH_{0,\mathbf{A}}(X):=\prod_{v\in\Omega_f(F)}\CH_0(X_v)$. Similarly, we define $\displaystyle F^1_\mathbf{A}(X)=\prod_{v\in\Omega_f(F)}F^1(X_v)$ and $\displaystyle F^2_\mathbf{A}(X)=\prod_{v\in\Omega_f(F)}F^2(X_v)$. 
\end{defn} 
When $F$ is not totally imaginary, it follows by \cite[Th\'{e}or\`{e}me 1.3]{Colliot-Thelene1993} that the group $\CH_{0,\mathbf{A}}(X)$ has a small ($2$-torsion) contribution from the infinite real places. In this article we will mainly be working over totally imaginary number fields, and hence we omit the more general definition. For more details see \cite[Def. 5.2]{Gazaki/Hiranouchi2021}. 

The local pairings induce a global pairing,
\[\langle\cdot,\cdot\rangle:\CH_{0,\mathbf{A}}(X)\times\Br(X)\rightarrow\Q/\Z,\] defined by $\langle(z_v)_v,\alpha\rangle=\sum_v\inv_v(\langle z_v,\iota_v^\star(\alpha)\rangle_v)$, where $\iota_v^\star$ is the pullback of $\iota_v: X_v\to X$. 
The short exact sequence of global class field theory, \[0\rightarrow\Br(F)\rightarrow\bigoplus_{v\in\Omega}\Br(F_v)\xrightarrow{\sum\inv_v}\Q/\Z\rightarrow 0,\] implies that the group $\CH_0(X)$ lies in the left kernel of $\langle\cdot,\cdot\rangle$. Thus, we obtain a complex, 
\begin{equation}\label{complex1}\CH_0(X)\stackrel{\Delta}{\longrightarrow}
\CH_{0,\mathbf{A}}(X)\rightarrow\Hom(\Br(X),\Q/\Z),
\end{equation} where $\Delta$ is the diagonal map. \autoref{locatoglobalconj} then predicts that the complex \eqref{complex1} becomes exact after passing to the completions; namely that 
the induced complex 
\begin{equation}\label{complex2}\widehat{\CH_0(X)}\stackrel{\Delta}{\longrightarrow}
\widehat{\CH_{0,\mathbf{A}}}(X)\rightarrow\Hom(\Br(X),\Q/\Z)
\end{equation} is exact.

 Next we consider the filtration $\Br(X)\supset\Br_1(X)\supset\Br_0(X)$ induced by the Hochschild-Serre spectral sequence, where $\Br_1(X):=\ker(\Br(X)\rightarrow\Br(X_{\overline{F}}))$ is the \textit{algebraic Brauer group} of $X$, and $\Br_0(X)=\img(\Br(F)\to\Br(X))$ are the constants. When $X$ has a $F$-rational point, the exactness of \eqref{complex2} is reduced to the exactness of the induced complex,
 \begin{equation}\label{complex3}\widehat{F^1(X)}\stackrel{\Delta}{\longrightarrow}
\widehat{F^1_\mathbf{A}(X)}\xrightarrow{\varepsilon}\Hom(\Br(X)/\Br_0(X),\Q/\Z). 
\end{equation} 
When $X$ is a product of elliptic curves more can be said. 
\begin{prop}\label{compatibility} Let $X=E_1\times\cdots\times E_d$ be a product of elliptic curves over a number field $F$. Let $G_F$ be the absolute Galois group of $F$. Suppose that the N\'{e}ron-Severi group $\NS(X_{\overline{F}})$ of the base change to the algebraic closure of $F$ has trivial $G_F$-action. Then the restriction of the map $\varepsilon:F^1(X)\to\Hom(\Br(X)/\Br_0(X),\Q/\Z)$ to the Albanese kernel $F^2(X)$ factors through the group $\Hom(\Br(X)/\Br_1(X),\Q/\Z)$. 
\end{prop}
\begin{proof} 
  We need to show that for every $z\in F^2(X)$ and every $\alpha\in \Br_1(X)/\Br_0(X)$, it follows $\langle z,\alpha\rangle=0$. A direct computation of the Hochshild-Serre spectral sequence gives an isomorphism $\Br_1(X)/\Br_0(X)\simeq H^1(F,\Pic(X_{\overline{F}}))$. The $G_F$-module $\Pic(X_{\overline{F}})$ fits into a (splitting) short exact sequence 
\[0\to \Pic^0(X_{\overline{F}}))\to \Pic(X_{\overline{F}}))\to\NS(X_{\overline{F}}))\to 0.\] Since $X$ is an abelian variety, the group $\NS(X_{\overline{F}})$ is torsion free, and hence a finitely generated free abelian group. Since we assumed that $G_F$ acts trivially on it, it follows that $H^1(F,\NS(X_{\overline{F}}))=0$. Thus, we obtain an isomorphism \[H^1(F,\Pic(X_{\overline{F}}))\simeq H^1(F,X(\overline{F}))\simeq \bigoplus_{i=1}^d H^1(k,E_i).\]

 For $i=1,\ldots,d$, let $\pr_i:X\to E_i$ be the projection. The map $\pr_i$ is proper and it induces a pushforward $\pr_{i\star}: \CH_0(X)\to\CH_0(E_i)$ on Chow groups and a pullback $\pr_i^\star:\Br(E_i)\to\Br(X)$ on Brauer groups. Since the Brauer-Manin pairing is defined just by evaluation, we have a commutative diagram
\[\begin{tikzcd} \CH_{0,\mathbf{A}}(X)\ar{d}{\pr_{i\star}}\ar{r}& \Hom(\Br(X),\Q/\Z)\ar{d}{\Hom(\pr_i^\star,\Q/\Z)}\\
\CH_{0,\mathbf{A}}(E_i)\ar{r}& \Hom(\Br(E_i),\Q/\Z).
\end{tikzcd}\]
That is, for $z\in\CH_0(X)$ and $\beta\in \Br(E_i)$ we have an equality
$\langle z,\pr_i^\star(\alpha)\rangle=\langle\pr_{i\star}(z),\alpha\rangle.$ 
Moreover, both homomorphisms preserve the two piece filtrations of  $\CH_0$ and $\Br$. The pullbacks $p_i^\star$
 induce a homomorphism \[\bigoplus_{i=1}^d \Br_1(E_i)/\Br_0(E_i)\xrightarrow{\oplus\pr_i^\star}\Br_1(X)/\Br_0(X).\] Our previous analysis shows that this is in fact an isomorphism. We compute
\[\langle z,\pr_i^\star(\alpha)\rangle=\langle\pr_{i\star}(z),\alpha\rangle=0.\] The vanishing follows, because $\pr_{i\star}(z)\in F^2(E_i)=0$, since $E_i$ is a curve.

\end{proof}

\begin{cor} Let $X=E\times E$ be the self-product of an elliptic curve over a number field $F$. Suppose that $E$ has complex multiplication defined over $F$ by the full ring of integers of a quadratic imaginary field $K$. Then the assumption of  \autoref{compatibility} is satisfied. 
\end{cor}
\begin{proof}
Because we assumed that $F$ contains the quadratic imaginary field $K$, it follows by \cite[p. (120), equation (10)]{Skorobogatov/Zharin2012} that $\NS(X_{\overline{F}})$ is a trivial $G_F$-module. 

\end{proof}

\autoref{compatibility} shows that the complex \eqref{complex3} induces a complex 

 \begin{equation}\label{complex4}\widehat{F^2(X)}\stackrel{\Delta}{\longrightarrow}
\widehat{F^2_\mathbf{A}(X)}\xrightarrow{\varepsilon}\Hom(\Br(X)/\Br_1(X),\Q/\Z). 
\end{equation} 
Combining \autoref{compatibility} with \cite[Prop. 5.6]{Gazaki/Hiranouchi2021} yields the following corollary. 
\begin{cor}\label{reduction} The exactness of \eqref{complex3} can be reduced to the exactness of \eqref{complex4}. 
\end{cor} 
We note that this corollary was claimed by the main author and T. Hiranouchi in \cite[Prop. 5.6]{Gazaki/Hiranouchi2021}. It was brought to our attention however that the argument there was not enough. \autoref{compatibility} was a key missing ingredient. We note that such a reduction cannot be made in general for other classes of varieties, for example del Pezzo surfaces, or $K3$ surfaces.



\vspace{3pt}
\section{Local Results}\label{Brauercomputations}
In this section we are going to prove theorem \eqref{main0} and also obtain some necessary local information in order to prove theorems \eqref{main1} and \eqref{main2} in the next section. 

\subsection{Proof of \autoref{main0}}\label{localprelims} In this subsection $k$ will denote a finite extension of $\Q_p$. We will often assume that $p$ is odd. Let $X$ be a smooth projective geometrically connected variety over $k$. We will denote by $F^2(X)_{\dv}$ the maximal divisible subgroup of the Albenese kernel $F^2(X)$ and by $F^2(X)_{\nd}$ the quotient $F^2(X)/F^2(X)_{\dv}$. We have a decomposition, \[
F^2(X)=F^2(X)_{\dv}\oplus F^2(X)_{\nd}. 
\]
The subgroup $F^2(X)_{\nd}$ is expected to be finite (cf. \cite[Conjecture 3.5.4]{Raskind/Spiess2000}, see also \cite[1.4(g)]{Colliot-Thelene1993}). When $X=E_1\times E_2$ is the product of two elliptic curves, this conjecture has been established in a large number of cases.
 In particular, the following are true.  
\begin{enumerate}
\item When $X$ has good reduction, the group $F^2(X)$ is $m$-divisible for every integer $m$ coprime to $p$ (\cite[Theorem 3.5]{Raskind/Spiess2000}). 
\item Suppose $p\geq 3$, $X$ has split semistable reduction and at most one of the curves has good supersingular reduction. Then the group  $F^2(X)_{\nd}$ is finite (\cite[Theorem 1.1]{Raskind/Spiess2000}, \cite[Theorem 1.2]{Gazaki/Leal2018}). 
\item Suppose $p\geq 3$, $k$ is unramified over $\Q_p$, $X$ has good reduction and at most one of the curves has good supersingular reduction. Then $F^2(X)_{\nd}=0$ (\cite[Theorem 1.4]{Gazaki/Hiranouchi2021}). 
\end{enumerate}

The finiteness of $F^2(X)_{\nd}$ yields an equality $F^2(X)_{\nd}=\widehat{F^2(X)}$. 
In the special case when $X$ has good reduction, it follows that the group $F^2(X)_{\nd}$ has $p$-power order. Suppose $F^2(X)_{\nd}=p^N$ for some $N\geq 0$. Then $F^2(X)_{\nd}\simeq F^2(X)/p^N$. 




 We are now ready to prove \autoref{main0}, which we restate here. 
\begin{theo}\label{mainmain0} Let $X=E_1\times E_2$ be a product of elliptic curves over a number field $F$. Suppose that the action of the absolute Galois group $G_F$ on the N\'{e}ron-Severi group $\NS(X_{\overline{F}})$ is trivial. Then, there is an infinite set $T$ of rational primes $p$ for which the group $\varprojlim\limits_{n}F^2_{\mathbf{A}}(X)/p^n$ vanishes. In particular, the complex \ref{complex2intro} is exact for every $p\in T$. If we further assume that at least one of the curves does not have potential complex multiplication, then the complement $S$ of $T$ is a thin set of primes. 
\end{theo} 
\begin{proof} The proof will be along the lines of \cite[Lemma 5.9]{Gazaki/Hiranouchi2021}. 
Let $p$ be an odd prime and suppose that for every place $v$ above $p$ the surface $X_v$ has good reduction. 
  It follows by \autoref{localprelims} (1) that the only factors of $\displaystyle F^2_{\mathbf{A}}(X)$ that might not be $p$-divisible correspond to those places $v$ above $p$ and to the places of bad reduction. Let $\Lambda=\{v\text{ place of bad reduction}\}$. Then  
   for every $n\geq 1$ we have, 
  \[F^2_{\mathbf{A}}(X)/p^n=\prod_{v|p}F^2(X)/p^n\times\prod_{v\in\Lambda}F^2(X)/p^n.\] Let $v\in\Lambda$ be a place of bad reduction. Then there exists a finite extension $L_v/F_v$ such that the base change $X_{L_v}$ has split semistable reduction. This means that either both elliptic curves $E_{iL_{v}}$ have good reduction, or if any of them has bad reduction, then it has split multiplicative reduction. In the former case, since $v\nmid p$, it follows by \cite[Theorem 3.5]{Raskind/Spiess2000} that the group $F^2(X_{L_v})$ is $p$-divisible. In the latter case, it follows by \cite[Theorem 1.2]{Gazaki/Leal2018}) that the group $F^2(X_{L_v})_{\nd}$ is finite. Write $N_v$ for its order. To unify notation, we set $N_v=1$ for the case of potentially good reduction. Consider the following positive integer,
  \[M=\prod_{v\in\Lambda}[L_v:F_v]\prod_{v\in\Lambda}N_v.\] We claim that for every prime $p$ which is coprime to $M$ and for every $v\in\Lambda$ the group $F^2(X_v)$ is $p$-divisible. Consider the projection $X_{L_{v}}\xrightarrow{\pi_{L_{v}/F_{v}}} X_{v}$, and let $\CH_0(X_{L_{v}})\xrightarrow{\pi_{L_{v}/F_{v}\star}} \CH_0(X_{v})$, and $\CH_0(X_{v})\xrightarrow{\pi_{L_{v}/F_{v}}^{\star}} \CH_0(X_{L_{v}})$ be the induced push-forward and pull back maps respectively. Then we have an equality $\pi_{L_{v}/F_{v}\star}\circ\pi_{L_{v}/F_{v}}^{\star}=[L_v:F_v]$. Since $p$ is coprime to $[L_v:F_v]$, it follows that the pull back map induces an injective map of $\F_p$-vector spaces 
  \[F^2(X_{v})/p\stackrel{\pi_{L_{v}/F_{v}}^{\star}}{\hookrightarrow} F^2(X_{L_{v}})/p.\] By our choice of $p$, it follows that the group $F^2(X_{L_{v}})$ is $p$-divisible, and hence so is $F^2(X_{v})$. 
  
  As a conclusion, for every prime $p\nmid M$ such that for every place $v$ above $p$ the surface $X_v$ has good reduction, we have an isomorphism 
  \[\varprojlim\limits_{n}F^2_{\mathbf{A}}(X)/p^n\simeq\prod_{v|p}\varprojlim\limits_{n}F^2(X_v)/p^n.\]
  
  

Let $T_0$ be the set of odd primes $p$ such that for every place $v$ above $p$ the abelian surface $X_v$ has good reduction, at most one of the curves $E_{1v}, E_{2v}$ has good supersingular reduction and $p\nmid M$. Then it follows by \autoref{localprelims} (2) that for every $p\in T_0$ and for every place $v|p$ the group $F^2(X)_{\nd}$ is finite of $p$-power order. Thus,  we have an isomorphism  
\begin{equation}\label{piso}
\varprojlim\limits_{n}F^2_{\mathbf{A}}(X)/p^n=\prod_{v|p}F^2(X_v)_{\nd}.\end{equation}
 If we further assume that for each place $v|p$ the extension $F_v/\Q_p$ is unramified, then it follows by \cite[Theorem 1.4]{Gazaki/Hiranouchi2021} that the group $F^2(X_v)_{\nd}$ vanishes. Consider the set \[T=T_0\setminus\{p:\text{ ramified prime}\}.\] Then for every $p\in T$, $\varprojlim\limits_{n}F^2_{\mathbf{A}}(X)/p^n=0$ and the complex \ref{complex2intro} is exact. 
 
 Note that the set $T_0$ is infinite, since it contains all primes $p$ such that for every place $v$ above $p$ both elliptic curves have good ordinary reduction. Since the set of ramified primes is finite, it follows that $T$ is an infinite set of primes. 

Next suppose that one of the curves does not have potential complex multiplication; without loss of generality, assume $\End(E_{1\overline{\Q}})=\Z$.  Then the set of primes $p$ which are such that the curve $E_{1v}$ has good supersingular reduction for every place $v$ above $p$ is a set of primes of density zero (cf.~\cite{Serre1981, Lang/Trotter1976}). We conclude that the complement $S$ of $T$ is a set of primes of density zero.

\end{proof}
\subsection{Computing the local Albanese kernel over ramified base fields}\label{localAlb} 
%
%

Throughout this subsection $E$ will be an elliptic curve defined over a finite extension $k$ of  $\Q_p$. We assume that $E$ has \textbf{good ordinary reduction}. We will denote by $\widehat{E}$ the formal group of $E$ and by $\overline{E}$ the reduction of $E$, which is an ordinary elliptic curve over the finite field $\F_k$. Moreover, we will denote by $e$ the absolute ramification index of $k$. The purpose of this subsection is to obtain some explicit information on the group $F^2(X)_{\nd}=\widehat{F^2(X)}$. 


%

\subsubsection{Decomposition of local points}\label{decompose}

 We first obtain some information on points $P\in E(k)$, which will be of use in \autoref{computations_section}.  
 Since the elliptic curve $E$ has good ordinary reduction, for every integer $n\geq 1$ we have a short exact sequence of $\Gal(\overline{k}/k)$-modules.
\begin{equation}\label{ses7}
0\rightarrow \widehat{E}[p^n]\rightarrow E[p^n]\rightarrow\overline{E}[p^n]\rightarrow 0,
\end{equation} where as $\Z$-modules both $\widehat{E}[p^n]$ and $\overline{E}[p^n]$ are isomorphic to $\Z/p^n$. The sequence \ref{ses7} is also known as the \textit{connected-\'{e}tale} exact sequence. In general this sequence does not split, but it does when $E[p^n]\subset E(k)$, and unconditionally on $k$ if $E$  has complex multiplication (\cite[A.2.4]{Serre89}). In particular it splits when $E$ satisfies the assumptions of \autoref{CM background}. 
We assume we are in the CM situation. We additionally suppose that $k$ is an unramified extension of $\Q_p$ and $\overline{E}[p]\subset\overline{E}(\F_k)$. 
 We consider the exact sequence 
\begin{equation}
\label{ses3}
0\rightarrow \widehat{E}(\mathfrak{m}_k)\rightarrow E(k)\xrightarrow{r}\overline{E}(\F_k)\rightarrow 0,
\end{equation} where $E(k)\xrightarrow{r}\overline{E}(\F_k)$ is the reduction map. The group $\overline{E}(\F_k)$ is finite, thus it decomposes as $\overline{E}(\F_k)\simeq \overline{E}(\F_k)\{p\}\oplus \overline{E}(\F_k)\{m\}$, for some large enough integer $m\geq 1$ coprime to $p$. Because $E$ has good reduction, it follows by the criterion of N\'{e}ron-Ogg-Shafarevich (\cite[Theorem 7.1]{Silverman2009} that the coprime-to-$p$ torsion subgroup of $E(k)$ is isomorphic to $\overline{E}(\F_k)\{m\}$. Moreover, since $\overline{E}$ is an ordinary elliptic curve and we assumed  $\overline{E}[p]\subset\overline{E}(\F_k)$, it follows that $\overline{E}(\F_k)\{p\}\simeq\overline{E}[p^{N_0}]$ for some integer $N_0\geq 1$. In particular, the group $\overline{E}(\F_k)\{p\}$ is cyclic of order $p^{N_0}$.  The splitting of \ref{ses7} implies that the map $E[p^{N_0}](k)\xrightarrow{r}\overline{E}[p^{N_0}](\F_k)$ is surjective. Since we assumed that $k/\Q_p$ is unramified, it follows by \cite[Theorem 6.1]{Silverman2009} that $\widehat{E}[p]=0$, and hence this map is an isomorphism. We conclude that the reduction map $r$ induces an isomorphism on torsion subgroups 
\[E(k)_{\tor}\simeq\overline{E}(\F_k)_{\tor},\] and hence the short exact sequence \eqref{ses3} has a natural splitting. Moreover, if $L/k$ is any finite extension, this splitting commutes with the restriction map $E(k)\xrightarrow{\res_{L/k}}E(L)$. 
It also induces a split short exact sequence 
\[0\rightarrow \widehat{E}(\mathfrak{m}_k)/p\rightarrow E(k)/p\xrightarrow{r}\overline{E}(\F_k)/p\rightarrow 0.\] This splitting can be made more explicit as follows. We have  an isomorphism, 
\[\overline{E}(\F_k)/p\simeq \overline{E}[p^{N_0}]/\overline{E}[p^{N_0-1}]\simeq \overline{E}[p].\]
 We fix a generator $\overline{P}_0$ of $\overline{E}[p^{N_0}](\F_k)$ and its unique lift  $P_0\in E[p^{N_0}](k)$. Let $P\in E(k)$. There exists an integer $c\in\{0,1,\ldots, p^{N_0}-1\}$  such that $r(P \mod p)-r(cP_0\mod p)=0$. This implies that $P$ has an expression as $P=\widehat{P}+cP_0+pQ$ for some points $Q\in E(k), \widehat{P}\in \widehat{E}(\mathfrak{m}_k)$.


\begin{notn}\label{localpointsplit} From now we will abuse notation and for a point $P\in E(k)$ we will denote also by $P$ its image in $E(k)/p$. The above discussion shows that $P$ has a unique decomposition as $P=(\widehat{P},\overline{P})$  where $\widehat{P}\in\widehat{E}(\mathfrak{m}_k)/p$ and $\overline{P}:=r(cP_0\mod p)\in \overline{E}(\F_k)/p$. It is clear by the construction that this decomposition is compatible with the restriction map $\res_{L/k}$ for finite extensions $L/k$. Namely, $\res_{L/k}(P)=(\res_{L/k}(\widehat{P}),\res_{\F_L/\F_k}(\overline{P}))$. 
\end{notn}
\subsubsection{Explicit isomorphism}\label{formalgroup} In this subsection  we consider a finite extension $L/k$ such that 
 $E[p^n]\subset E(L)$ for some $n\geq 1$. It follows by \cite[Section 4]{Raskind/Spiess2000} (see also \cite[Theorem 4.2]{Hiranouchi2014}, \cite[Theorem 3.4]{hiranouchi/hirayama}) that we have an isomorphism 
\begin{equation}\label{modpiso}
F^2(X_L)_{\nd}/p^n=F^2(X_L)/p^n\simeq K_2(L;E)/p^n\simeq \Z/p^n\Z. 
\end{equation} 
Let $N=\max\{n\geq 1:E[p^n]\subset E(L)\}$. If we further assume that the extension $L(E[p^{N+1}])/L$ has wild ramification, it follows by \cite[Theorem 3.14]{Gazaki/Leal2018} that we have an isomorphism
\begin{equation}\label{GazLealiso}
\widehat{F^2(X_L)}=F^2(X_L)_{\nd}=F^2(X_L)/p^N\simeq\Z/p^N.
\end{equation}
The case of interest to us is when $N=1$. To prove \autoref{main1}, the above information is enough. However, in order to construct global $0$-cycles and prove \autoref{main2}, we need to recall the explicit construction of the isomorphism $K_2(L;E_L)/p\simeq \Z/p\Z$.  We are particularly interested in describing necessary and sufficient conditions for a symbol $\{P,Q\}_{L/L}\in K_2(L;E_L)/p$ to be nontrivial. 

The group $K_2(L;E_L)$ is defined to be the quotient of the Mackey product $(E_L\otimes^M E_L)(L)$ (cf. \cite[Section 2.1]{Gazaki/Hiranouchi2021}) modulo Weil reciprocity. It follows by \cite{Raskind/Spiess2000} (see also \cite[Theorem 4.2]{Hiranouchi2014}) that when $E[p]\subset E(L)$ we have an isomorphism 
\[K_2(L;E_L)\simeq (E_L\otimes^M E_L)(L)/p.\] The Mackey functor $E_L/p$ has a direct decomposition 
$E_L/p\simeq \widehat{E_L}/p\oplus [E_L/\widehat{E_L}]/p,$ induced by the splitting short exact sequences 
\[0\to \widehat{E_L}(m_{L'})/p\to E_L(L')/p\to E_L(\F_{L'}/p)\to 0,\] where $L'/L$ is any finite extension. For the definition of the Mackey functor $[E_L/\widehat{E_L}]$ we refer to \cite[Proof of Theorem 3.14]{Gazaki/Leal2018}. It follows by \cite[Lemma 3.4.2]{Raskind/Spiess2000} and \cite[Proof of Theorem 3.1.4]{Gazaki/Leal2018} respectively that the Mackey products $([E_L/\widehat{E_L}]\otimes^M [E_L/\widehat{E_L}])(k)$, $(\widehat{E_L}\otimes^M [E_L/\widehat{E_L}])(L)/p,$ and $ ([E_L/\widehat{E_L}]\otimes^M E_L)(k)/p$ all vanish. Thus we have an isomorphism 
\[K_2(L;E_L)/p\simeq (\widehat{E_L}\otimes^M \widehat{E_L})(k)/p=(\widehat{E_L}/p\otimes^M \widehat{E_L}/p)(k).\] 
The upshot is the following, which will be used in \autoref{computations_section}. Let $k/\Q_p$ be a finite unramified extension like in the previous subsection and let $P\in E(k)/p$. Consider the restriction of $P$ in $E(L)/p$, which for simplicity we will denote again by $P$. Consider the decomposition $P=(\widehat{P},\overline{P})$ as in notation \eqref{localpointsplit}. Let $\widehat{Q}\in\widehat{E_L}(\mathfrak{m}_L)/p$. Then we have, 
\begin{equation}\label{splitformal}\{P,\widehat{Q}\}_{L/L}=\{\widehat{P},\widehat{Q}\}_{L/L}\in K_2(L;E_L)/p.\end{equation}

Next we analyze the Mackey functor $\widehat{E_L}/p$. 
\black  The inclusion $E[p]\subset E(L)$ implies $\mu_p\subset L^\times$, and $\widehat{E}[p]\subset \widehat{E_L}(\mathfrak{m}_L)$. We fix a primitive $p$-th root of unity $\zeta_p\in L^\times$ and a non-canonical isomorphism of $\Gal(\overline{L}/L)$-modules 
$\widehat{E}[p]\simeq\Z/p\simeq\mu_{p}.$ This induces an isomorphism 
\[H^1(L,\widehat{E_L}[p])\simeq H^1(L,\mu_p)=L^\times/L^{\times p}.\] We consider the connecting homomorphism of the Kummer sequence for $\widehat{E_L}$ (cf. \cite[Lemma 3.5]{Gazaki/Hiranouchi2021})
\[\delta_L:\widehat{E_L}(L)/p\hookrightarrow H^1(L,\widehat{E_L}[p]).\]
 Let $U_{L}^1$ be the group of $1$-units of $L^\times$ with its natural filtration $\{U_{L}^i=1+\mathfrak{m}_{L}^i\}_{i\geq 1}$ and define 
\[\overline{U}_{L}^i=\img(U_{L}^i\rightarrow L^\times/L^{\times p}),\;\;i\geq 1.\]
It then follows by \cite[Theorem 2.1.6]{kawachi2002}, (see also \cite[Prop. 3.13]{Gazaki/Hiranouchi2021}) that the homomorphism $\delta_L$ induces an isomorphism 
\begin{equation}\label{kawachi}\delta_L:\widehat{E_L}(L)/p\xrightarrow{\simeq}\overline{U}_{L}^1,  
\end{equation} and the same holds over any finite extension $L'/L$. In fact we have an isomorphism as Mackey functors  $\widehat{E_L}/p\simeq \overline{U}_{L}^1$. 

Additionally, $\delta_L$ is a filtered isomorphism in the following sense. The formal group $\widehat{E_L}(\mathfrak{m}_L)$ has a natural filtration 
$\widehat{E_L}^i(L):=\widehat{E_L}(\mathfrak{m}_{L}^i)$, for $i\geq 1$. This induces a filtration on the group $\widehat{E_L}(\mathfrak{m}_{L})/[p]=\widehat{E_L}^1(L)/p$ by defining (cf. \cite[Definition 3.10]{Gazaki/Hiranouchi2021}),
\[\mathcal{D}^i:=\frac{\widehat{E_L}^i(L)}{[p]\widehat{E_L}(L)\cap \widehat{E_L}^i(L)},\;\;i\geq 1.\] Then the isomorphism \ref{kawachi} has the property that $\delta_L(\mathcal{D}^i)\subset\overline{U}_L^i$, for all $i\geq 1$. Moreover, it is functorial with respect to norm maps associated with finite extensions $L'/L$. 

Since $\mu_p\subset L^\times$, we have isomorphisms $H^2(L,\mu_p\otimes\mu_p)\simeq H^2(L,\mu_p)\simeq\Br(L)_p\simeq\Z/p$. For $a,b\in L^\times/L^{\times p}$ write $(a,b)_p\in\Br(L)_p$ for the cyclic symbol algebra generated by $a,b$. Then we have an isomorphism 
\begin{eqnarray}\label{sp}
&& s_p: (\widehat{E_L}/p\otimes^M \widehat{E_L}/p)(k)\rightarrow  H^2(L,\mu_p)=\Br(L)_p\nonumber\\
&& \{\widehat{P},\widehat{Q}\}_{L/L}\mapsto (\delta(\widehat{P}),\delta(\widehat{Q}))_p,
\end{eqnarray} known as \textit{the generalized Galois symbol}. For a proof of the injectivity of $s_p$ see for example \cite[Theorem 4.2]{Hiranouchi2014}.

\subsubsection{Criterion for nontriviality}\label{iff}

It follows that a symbol $\{\widehat{P},\widehat{Q}\}_{L/L}$ is nonzero in $K_2(L;E_L)/p$ if and only if $\delta_L(\widehat{P})$ is not in the image of the norm map \[N:L\left(\delta_L(\widehat{Q})^{1/p}\right)^\times\rightarrow L^\times,\] or equivalently if and only if $\widehat{P}$ is not in the image of the norm map 
\[N:\widehat{E_L}\left(L\left(\frac{1}{p}\widehat{Q}\right)\right)/p\rightarrow \widehat{E_L}(L)/p.\] Here we denoted by $L\left(\frac{1}{p}\widehat{Q}\right)$ the smallest Galois extension of $L$ over which there exists a point $\widehat{Q}'$ such that $p\widehat{Q}'=\widehat{Q}$. The same equivalence holds with the roles of $\widehat{P},\widehat{Q}$ interchanged. 

\vspace{5pt}
\section{Self-product of an elliptic curve with complex multiplication}\label{CMsection} 
 In this subsection we are going to prove theorems \eqref{main1} and \eqref{main2}. Throughout this section we are making the following assumption.
\begin{ass}\label{Epassume} $K$ will be a quadratic imaginary field of class number one. $E$ will be an elliptic curve over $\Q$ or $K$ with CM by $\mathcal{O}_K$ and $p\geq 5$ will be an ordinary prime for $E$. Moreover, we fix a factoring $p=\pi\overline{\pi}$ into prime elements of $\mathcal{O}_K$ such that  the  endomorphism $\pi:E_K\rightarrow E_K$ reduces to the Frobenius automorphism $\phi_p:\overline{E_p}\rightarrow\overline{E_p}$. 
\end{ass}

\subsection{Local $0$-cycles orthogonal to the Brauer group}\label{Brauercomplement}

 
 We are now ready to prove \autoref{main1}, which we restate here. 
 \begin{theo}\label{mainmain1}
 Let $X=E\times E$. Suppose there exists a finite Galois extension  $F_0/K$  of degree $n<p-1$ such that there is a unique inert place $w$ of $F_0$ above $\pi$ with the property that $p$ divides $|\overline{E_p}(\F_w)|$, where $\F_w$ is the residue field of $F_0$ at $w$.
  Consider the extension $L:=F_0\cdot K(E[\pi])$. Then the group $\displaystyle\varprojlim\limits_n F^2_{\mathbf{A}}(X_L)/p^n$ has the following properties. 
  \begin{enumerate}
  \item It is equal to $\displaystyle\prod_{v|p} F^2(X_{L_v})_{\nd}$ and  isomorphic to $\Z/p\oplus\Z/p$.
  \item It is orthogonal to the transcendental Brauer group $\displaystyle\frac{\Br(X_L)\{p\}}{\Br_1(X_L)\{p\}}$. 
\end{enumerate}  
 \end{theo} 
  \begin{proof} The proof involves three main steps. 
  
  \textbf{Claim 1:} There is an isomorphism $\displaystyle\prod_{v|p}F^2(X_{L_v})_{\nd}\simeq\Z/p\oplus\Z/p$. 

 It follows by \cite[Corollary 5.20 (ii), (iii)]{Rubin1999} that $K(E[\pi])/K$ is a Galois extension of degree $p-1$ and it is totally ramified above $\mathfrak{p}=(\pi)$. Since $F_0/K$ is unramified above $\pi$, the extensions $F_0, K(E[\pi])$ are totally disjoint and hence  $\Gal(L/K)\simeq\Gal(F_0/K)\times\Gal(K(E[\pi])/K)$. In particular, $L/K$ is a Galois extension of degree $n(p-1)<(p-1)^2$, and there exists a unique place $v$ of $L$ above $\pi$ with ramification index $p-1$ and residue field $\F_{w_0}$ of degree $n$ over $\F_p$. Moreover, the extension $L/\Q$ is Galois, and hence there exists a unique place $\overline{v}$ above $\overline{\pi}$ with the same properties and these are precisely the places of $L$ that lie above $p$. The extensions $L_v$ and $L_{\overline{v}}$ are the same. It is enough therefore to prove an isomorphism $\widehat{F^2(X_{L_v})}\simeq\Z/p\Z$.
 
  We will simply denote by $E_v$ (resp. $E_{\overline{v}}$) the base change $E_{L_v}$ (resp.  $E_{L_{\overline{v}}}$) which is an elliptic curve with good ordinary reduction over the $p$-adic field $L_v/\Q_p$. We will denote by $\mathfrak{m}_v$ the maximal ideal of $\mathcal{O}_{L_v}$, $\widehat{E_v}$ will be the formal group of $E_v$ and $\overline{E_v}$ the reduction.
We claim that $E_v[p]\subset E_v(L_v)$. As noted in \autoref{CM background}, the subgroup $E_{\Q_p}[\pi]$ of $E_{\Q_p}[p]$ coincides with the $p$-torsion of the formal group $\widehat{E_{\Q_p}}[p]$. It follows that $E_v[\pi]$ coincides with $\widehat{E_v}[p]$, and hence it is $L_v$-rational. 
 Since the residue field of $L_v/\Q_p$ is $\F_{w_0}$, the assumption that $p$ divides $|\overline{E_p}(\F_{w_0})|$ implies that $\overline{E_p}[p]\subseteq\overline{E_p}(\F_{w_0})$. The inclusion $E_v[p]\subset E_v(L_v)$ then follows by the splitting short exact sequence 
\[0\rightarrow \widehat{E_v}[p]\rightarrow E_v[p]\rightarrow\overline{E_p}[p]\rightarrow 0.\] 

The desired isomorphism $\widehat{F^2(X_{L_v})}\simeq\Z/p\Z$ follows by \eqref{GazLealiso}.  More precisely, the above argument shows that we have an inequality $N\geq 1$. We claim that $N=1$. It is enough to show that the extension $L_v(E[p^2])/L_v(E[p])$ has wild ramification. But this is clear, since $L_v(E[p^2])$ contains $L_v(\mu_{p^2})$, which has absolute ramification index dividing $p(p-1)$, while $L_v=L_v(E[p])$ has absolute ramification index $p-1$.    

   \textbf{Claim 2:} There is an isomorphism $\displaystyle\varprojlim_{n} F^2_\mathbf{A}(X_L)/p^n\simeq\prod_{v|p}F^2(X_{L_v})_{\nd}$. 
   
We use the notation introduced in the proof of \autoref{mainmain0}. Namely, let \[M=\prod_{v_0\in\Lambda}n_{v_0}\prod_{v_0\in\Lambda}N_{v_0},\] where $\Lambda$ is the set of all bad reduction places of $E_L$, $n_{v_0}$ is the degree of an extension $L'_{v_0}/L_{v_0}$ over which the elliptic curve $E_{L'_{v_0}}$ attains split semistable reduction and $N_{v_0}=1$ in the potentially good reduction case, $N_{v_0}=|F^2(X_{L'_{v_0}})_{\nd}|$ otherwise. To prove the claim, it is enough to show that $p$ is coprime to $M$, because then the proof of \autoref{mainmain0} gives us the desired isomorphism (cf. \ref{piso}).  
 Since the elliptic curve $E$ has complex multiplication, it follows by \cite[Theorem 7]{Serre/Tate1968} that the set $\Lambda$ only contains places of potentially good reduction. Thus, $N_{v_0}=1$ for every $v_0\in\Lambda$. 
 It remains to show that for each $v_0\in\Lambda$ there exists a finite extension $L'_{v_0}/L_{v_0}$ of degree coprime to $p$ over which $E$ attains good reduction. Let $v_0\in\Lambda$ and $l$ be the rational prime lying below $v_0$. It follows by \cite[Corollary 5.22]{Rubin1999} that there exists an elliptic curve $E'$ defined over $\Q$ such that $E'$ has good reduction at $l$ and $E_{\overline{\Q}}\simeq E'_{\overline{\Q}}$. But any two elliptic curves over $\Q$ become isomorphic over a degree $6$ extension. Since we assumed that $p\geq 5$, the claim follows. 

 \textbf{Claim 3:} The group $\displaystyle\frac{\Br(X_L)\{p\}}{\Br_1(X_L)\{p\}}$ is trivial. 

The group $\Br(X_L)/\Br_1(X_L)$ is finite by \cite[Theorem 1.1]{Skorobogatov/Zharin2008}. More precisely, if $G_L=\Gal(\overline{L}/L)$, then Skorobogatov and Zarhin 
(\cite[Proposition 3.3]{Skorobogatov/Zharin2012}) showed an isomorphism,
\[\frac{\Br(X_L)_p}{\Br_1(X_L)_p}\simeq\frac{\Hom_{G_L}(E_L[p],E_L[p])}{(\Hom(E_L,E_L)/p)^{G_L}}.\] Since  $E_L$ has CM defined over $L$, $(\Hom(E_L,E_L)/p)^{G_L}\simeq(\Z/p\Z)^2$. We claim that this is equal to  the numerator $\Hom_{G_L}(E_L[p],E_L[p])$. Since by assumption $n<p-1$, we have proper field extensions $K\subsetneq L\subsetneq K(E[p])$. It follows that $E_L[\pi]$ is the unique nonzero submodule of $E_L[p]$ with a trivial $G_L$-action. Hence, any $G_L$-equivariant homomorphism $f:E_L[p]\rightarrow E_L[p]$ must send $E_L[\pi]$ to itself. This is precisely a subgroup of rank $2$ of $\Hom(E_L[p],E_L[p])\simeq(\Z/p\Z)^4$ containing $(\Hom(E_L,E_L)/p)^{G_L}$, and hence it must be equal to it. 

 \end{proof} 
 
\begin{defn} Let $F_0/K$ be an extension satisfying the assumptions of \autoref{mainmain1}. Then $\F_{w_0}\simeq\F_{p^n}$. We will say that $F_0$ is \textit{minimal} if it has the minimum possible degree. That is, if $p||\overline{E_p}(\F_{p^n})|$ and $n\geq 1$ is the smallest possible with this property. 
\end{defn} 
When we construct global $0$-cycles in the next subsection, we will often start with a minimal totally real extension $F_0/\Q$ and apply \autoref{mainmain1} for $F_0\cdot K$. When $p||\overline{E_p}(\F_{p})|$, $F_0=\Q$. In all other cases a minimal $F_0$ is not uniquely determined. The following proposition shows that the extension $L=F_0\cdot K(E[\pi])$ constructed in \autoref{main1} is in some sense minimal over which interesting things happen at the places above the ordinary prime $p$. 

\begin{prop}\label{minimality} Let $(E,p)$ be a pair satisfying the assumptions of \autoref{mainmain1}. Let $F_0/\Q$ be an extension as in \autoref{mainmain1} of minimal degree. Let $\Q\subset F\subsetneq L=F_0\cdot K(E[\pi])$ be any intermediate extension. Then $\displaystyle\prod_{w\in\Omega_f(F),w|p}F^2(X_{F_{w}})_{\nd}=0$. 
\end{prop}
\begin{proof} Since $F\subsetneq L=F_0\cdot K(E[\pi])$, any such field $F$ has at most two places $w$ above $p$ of ramification index dividing $p-1$. Moreover, since $F$ is properly contained in $L$, and the fields $F_0$ and $K(E[\pi])$ are minimal with respect to their defining properties, it follows that for every place $w$ of $F$ above $p$, $E[p]\not\subset E(F_w)$. This observation reduces the proposition to the following local claim.

\textbf{Claim:} Let $k/\Q_p$ be a finite extension such that $E[p]\not\subset E(k)$. Then $K_2(k;E_k)/p=0$. 

Consider the finite extension $k_1=k(E[p])$ and notice that $k_1/k$ is a nontrivial extension of degree coprime to $p$. We have a commutative diagram 
\[\begin{tikzcd} K_2(k;E_k)/p\ar{r}{s_p}\ar{d}{\res_{k_1/k}} & H^2(k,E_k[p]^{\otimes 2})\ar{d}{\res_{k_1/k}}\\
K_2(k_1;E_{k_1})/p \ar{r}{s_p} & H^2(k_1,E_{k_1}[p]^{\otimes 2})
\end{tikzcd},\] where $s_p$ is the \textit{Galois symbol} map (cf. \cite[Definition 2.13]{Gazaki/Hiranouchi2021}) and $\res_{k_1/k}$ are the restriction maps (cf. \cite[p. 7]{Gazaki/Hiranouchi2021}). The bottom horizontal map is injective (cf. \cite[Theorem 4.2]{Hiranouchi2014}). Moreover, $N_{k_1/k}\circ\res_{k_1/k}=[k_1:k]$, and hence the restriction maps are also injective. This forces the top horizontal map to be injective, and hence it is enough to show that the image of the Galois symbol $K_2(k;E_k)/p\xrightarrow{s_p}H^2(k,E_k[p]^{\otimes 2})$ is zero. 

The image of $s_p$ can be computed by \cite[Theorem 1.1]{Gazaki2017}. Namely, under the Tate duality perfect pairing 
\[H^2(k,E_k[p]^{\otimes 2})\times\Hom_{\Gal(\overline{k}/k)}(E_k[p],E_k[p])\to\Z/p,\]
 the orthogonal complement of $\img(s_p)$ consists precisely of those $\Gal(\overline{k}/k)$-  homomorphisms $E_k[p]\xrightarrow{f} E_k[p]$ that lift to a homomorphism of finite flat group schemes $\mathcal{E}_k[p]\xrightarrow{\tilde{f}}\mathcal{E}_k[p]$, where $\mathcal{E}_k$ is the N\'{e}ron model of $E_k$ over $\Spec(\mathcal{O}_k)$. Since $E_k$ has good ordinary reduction, it follows by \cite[Proposition 8.8]{Gazaki2017} that these are exactly the homomorphisms $E_k[p]\xrightarrow{f} E_k[p]$ that satisfy $f(\widehat{E_k}[p])\subset \widehat{E_k}[p]$. Thus, to show $\img(s_p)=0$, it is enough to show that every $\Gal(\overline{k}/k)$-equivariant homomorphism $E_k[p]\xrightarrow{f} E_k[p]$ satisfies this property. Since $E_k$ has complex multiplication defined over $k$, it follows that $E_k[p]\simeq\widehat{E_k}[p]\oplus\overline{E_k}[p]$ as $\Gal(\overline{k}/k)$-modules. Since at least one of the cyclic submodules is nontrivial, the claim follows. 

\end{proof}
 
 In \autoref{CM background} we saw many examples when $F_0$ can be taken to be $\Q$. We will focus more on this case in the next subsection.
 \blue 
 
\black

 \vspace{1pt}

\vspace{2pt}
\subsection{Constructing global $0$-cycles}\label{computations_section}
In this last subsection we investigate the exactness of the complex \eqref{complex2intro} in the situation considered in \autoref{mainmain1}. That is, we are looking for criteria that allow us to lift the group $\displaystyle\prod_{v|p}F^2(X_{L_v})_{\nd}\simeq(\Z/p)^2$ to global $0$-cycles in $F^2(X)$.


\begin{lem}\label{independence} Let $K=\Q(\sqrt{D})$ with $D\in\{-1,-2,-3,-7,-11,-19,-43,-67,-163\}$. Let $E$ be an elliptic curve over $\Q$ such that $E_K$ has complex multiplication by $\mathcal{O}_K$. Let $F/\Q$ be a totally real extension and suppose that the Mordell-Weil group $E(F)$ has positive rank. Let $P\in E(F)$ be a point of  infinite order. Write $\mathcal{O}_K=\Z[\omega_D]$. Then the points $P,\omega_D(P)\in E(K\cdot F)$ are $\Z$-linearly independent. 
\end{lem}
\begin{proof} We will simply write $\omega$ instead of $\omega_D$. Note that $\omega(P)\in E(K\cdot F)\setminus E(F)$. Let $\mu(x)=x^2-bx+c\in\Z[x]$ be the minimal polynomial of $\omega$. Let $\overline{\omega}$ be the conjugate of $\omega$, so the equalities $\omega+\overline{\omega}=b$ and $\omega\overline{\omega}=c$ hold. 
Suppose for contradiction that $P,\omega(P)$ are $\Z$-linearly dependent.  Then we can find $n,m\in\Z$ such that $nP+m\omega(P)=0$. We may assume that $n,m$ are relatively prime, and hence we can find $x,y\in\Z$ such that $nx+my=1$. The relation $nP+m\omega(P)=0$ implies that $nxP+mx\omega(P)=0=nyP+my\omega(P)$. Applying the endomorphism $\overline{\omega}$ to the first equality yields $nx\overline{\omega}(P)+mxcP=0$, and hence $nx\omega(P)-(mxc+nxb)P=0$. Adding the latter to the equation $nyP+my\omega(P)=0$ gives 
\[\omega(P)=(nxb-ny+mxc)P\in E(F),\] which is a contradiction.

\end{proof} 

\begin{rem}\label{congruence} Suppose $D\in\{-3,-7,-11,-19,-43,-67,-163\}$. Then $\mathcal{O}_K=\Z[\omega_D]$, where $\displaystyle\omega_D=-\frac{1}{2}+\frac{\sqrt{-D}}{2}$. The minimal polynomial of $\omega_D$  is $x^2+x+\displaystyle\frac{1+D}{4}$ so that in the notation of \autoref{independence} we have $b=-1$, $c=\displaystyle\frac{1}{4}+\frac{D}{4}$. Let $p$ be an odd prime, $p\neq  -D$. Then the system of equations 
$\left\{\begin{array}{c}
x^2=c\\
2x=-1
\end{array}\right.$ has no solution over $\F_p$. For, if $x$ were a solution, then the solution would be $x=-1/2\in\F_p$, giving $4c\equiv 1\mod p\Leftrightarrow D\equiv 0\mod p$. This simple observation will be used in the proof of the following \autoref{mainmain2}. 
\end{rem}

\vspace{1pt}



\subsection*{Proof of lifting when a ``good" rational point exists}
In this section we give a first criterion to lift the local $0$-cycles to global. Our criterion relies only on the existence of a ``good" rational point $P$. Before proving our main results, we briefly describe our methodology. 

From now on we assume we are in the set-up of \autoref{mainmain1} with the extension $F_0$ minimal. 
The base change $X_L$, where $L=F_0\cdot K(E[\pi])$, contains a nontrivial $L$-rational $\pi$-torsion point $A\in E[\pi](L)$. Our strategy is to examine when this torsion point can be used to construct global $0$-cycles.

 We say that a point $P\in E(F_0)$ is good if the $0$-cycle $z_1=[A,P]-[A,O]-[O,P]+[O,O]\in F^2(X_L)$ induces a nontrivial element  $(z_{1v},z_{1\overline{v}})\in\displaystyle F^2(X_{L_v})_{\nd}\oplus F^2(X_{L_{\overline{v}}})_{\nd}$. When this happens, we show that the $0$-cycle $z_2=[A,\omega(P)]-[A,O]-[O,\omega(P)]+[O,O]$ induces a second $\F_p$-linearly independent element of $\displaystyle F^2(X_{L_v})_{\nd}\oplus F^2(X_{L_{\overline{v}}})_{\nd}$. To find an explicit condition that guarantees 
$(z_{1v},z_{1\overline{v}})\neq(0,0)$, we will use the nonvanishing criterion from \autoref{iff}. From now on we identify the $0$-cycle $z_1$ with the symbol $\{A,P\}_{L/L}\in K_2(L;E)$ and we denote by $\{A_v,P_v\}_{L_v/L_v}$ the corresponding element of $K_2(L_v;E_v)$. We will abuse notation and write $\{A_v,P_v\}_{L_v/L_v}$ for the class of this symbol modulo $pK_2(L_v;E_v)$. We recall from \eqref{splitformal} that we have an equality \[\{A_v,P_v\}_{L_v/L_v}=\{\widehat{A_v},\widehat{P}_v\}_{L_v/L_v}=\{A_v,\widehat{P}_v\}_{L_v/L_v},\] where $P_v=(\widehat{P}_v,\overline{P_v})$ is the decomposition of $P$ modulo $pE(L_v)$ as in \eqref{localpointsplit}. Here we used the fact that $A_v\in\widehat{E_v}[p]$. 

The following \autoref{mainmain2} develops the above idea when $F_0=\Q$. In this case we are looking for a point $P\in E(\Q)$ inducing the desired symbol $\{A,P\}_{L/L}$. For such a point we will denote by $P_{\Q_p}=(\widehat{P}_{\Q_p},\overline{P}_{\Q_p})\in \widehat{E}(\Q_p)/p\oplus\overline{E_p}(\F_p)$ the corresponding local point.


\begin{theo}\label{mainmain2} Suppose $p$ is a prime and $E$ an elliptic curve over $\Q$ satisfying the assumptions of \autoref{mainmain1}. Assume further that $|\overline{E}_p(\F_p)|=p$,  and consider the extension $L=K(E[\pi])$ constructed in \autoref{mainmain1}. Suppose that the Mordell-Weil group $E(\Q)$ has positive rank and there exists a global point $P\in E(\Q)$ of infinite order such that the induced local point $P_{\Q_p}\in E_{\Q_p}(\Q_p)$ has the property that $\widehat{P}_{\Q_p}\in\widehat{E_{\Q_p}}(\Q_p)/p$ is  nontrivial. 
Then the global $p$-torsion $0$-cycles $z_1,z_2\in F^2(X_L)$ induced by $A, P$ and $A,\omega(P)$ respectively  lift the group $\displaystyle\prod_{v|p}F^2(X_{L_v})_{\nd}\simeq\Z/p\oplus\Z/p$. 
\end{theo} 
\begin{proof} The assumption $|\overline{E}_p(\F_p)|=p$ implies that $K=\Q(\sqrt{D})$ with $D\equiv 1\mod 4$. Let $\mu(x)=x^2+x+c\in\Z[x]$ be the minimal polynomial of $\omega=\omega_D$ over $\Q$. 

\textbf{Claim 1:} The symbol $\{A_v,\widehat{P}_v\}_{L_v/L_v}\in K_2(L_v;E_v)/p$ is nontrivial. 


To prove the claim, we consider the Kummer extension $k_0=L_v\left(\delta(A_v)^{1/p}\right)=L_v\left(\frac{1}{p}A_v\right)$. Recall from \autoref{iff} that the symbol $\{A_v,\widehat{P}_v\}_{L_v/L_v}$ is nontrivial if and only if $\delta(\widehat{P}_v)$ is not in the image of the norm map 
\[N_{k_0/L_v}:k_0^\times\rightarrow L_v^\times.\]
We consider the filtration $\{\widehat{E_v}^i(L_v)\}_{i\geq 1}$ of $\widehat{E_v}(L_v)$ and the induced filtration $\{\mathcal{D}^i\}_{i\geq 1}$ of  $\widehat{E_v}(L_v)/p$ (cf. \autoref{formalgroup}). We claim that $A_v\in\widehat{E_v}^1\setminus\widehat{E_v}^2$ and it therefore induces a nontrivial element of $\mathcal{D}^1/\mathcal{D}^2$. For, let $v(A_v)$ be the valuation of $A_v$. Since $A_v$ is an element of $\widehat{E}(L_v)$ of exact order $p$, it follows by \cite[IV.6, Theorem 6.1]{Silverman2009} that
\[1\leq v(A_v)\leq\frac{v(p)}{p-1}.\] The extension $L_v/\Q_p$ is totally ramified of degree $p-1$, and hence it follows that $\frac{v(p)}{p-1}=1$ yielding the desired equality $v(A_v)=1$.

 We note that the integer  $\frac{v(p)}{p-1}=1$ is precisely the integer $t_0(\phi)$ defined in \cite[page 11, Lemma 3.5]{Gazaki/Hiranouchi2021} where $\phi=[p]$ is the $p$-isogeny on $\widehat{E_v}$, which is of height one. Using the filtered isomorphism \eqref{kawachi}, we deduce that $\delta(A_v)\in\overline{U}_{L_v}^1\setminus\overline{U}_{L_v}^2$. It follows by \cite[Lemma 2.1.5]{kawachi2002} (see also \cite[Lemma 3.5]{Gazaki/Hiranouchi2021}) that the extension $k_0/L_v$ is totally ramified of degree $p$ with the jump in the ramification filtration of $\Gal(k_0/L_v)$ happening at $s=p-1$. It then follows by \cite[V.3, Corollary 7]{Serre1979local} that there exists a unit $u\in\overline{U}_{L_v}^{p-1}$ such that the symbol algebra $(\delta(A_v),u)_p$ is nontrivial in $\Br(L_v)_p$. In fact, we can conclude that $u\in\overline{U}_{L_v}^{p-1}\setminus\overline{U}_{L_v}^p$. For, it follows by \cite[V.3, Corollaries 2 \& 3]{Serre1979local} that $(\delta(A_v),u')_p=0$ for every $u'\in\overline{U}_{L_v}^p$. The next key observation is that \[(\delta(A_v),u)_p\neq 0, \text{ \textbf{for every unit} } u\in\overline{U}_{L_v}^{p-1}\setminus\overline{U}_{L_v}^p.\] This is because the residue field of $L_v$ is $\F_p$, and therefore we have an isomorphism (cf. \cite[Lemma 3.4]{Gazaki/Hiranouchi2021}) $\overline{U}_{L_v}^{p-1}/\overline{U}_{L_v}^{p}\simeq\F_p$. Thus, in order to prove the claim, it is enough to show $\delta(\widehat{P}_v)\in \overline{U}_{L_v}^{p-1}\setminus\overline{U}_{L_v}^p$, or equivalently that $v(\widehat{P}_v)=p-1$. Consider the inclusion
 \[\res_{L_v/\Q_p}:\widehat{E_{\Q_p}}^1(\Q_p)/p\hookrightarrow\widehat{E_v}^{p-1}(L_v)/p,\] sending $\widehat{P}_{\Q_p}$ to $\widehat{P}_{v}$. This map is injective, because we have an equality $N_{L_v/\Q_p}\circ\res_{L_v/\Q_p}=[L_v:\Q_p]=p-1$, which is coprime to $p$. 
Our last claim will then follow  once we show that $\widehat{P}_{\Q_p}\in \widehat{E_{\Q_p}}^1(\Q_p)\setminus\widehat{E_{\Q_p}}^2(\Q_p)$. Let us assume otherwise that $\widehat{P}_{\Q_p}\in\widehat{E_{\Q_p}}^2(\Q_p)$, and hence $\widehat{P}_{v}\in\widehat{E_v}^{2(p-1)}$. Because $2(p-1)>p$, it follows by the filtered isomorphism \ref{kawachi} and \cite[Lemma 3.4]{Gazaki/Hiranouchi2021} that $\widehat{P}_{v}$ (and hence also $\widehat{P}_{\Q_p}$) is a multiple of $p$, which contradicts our assumption that $\widehat{P}_{\Q_p}\in \widehat{E_{\Q_p}}^1(\Q_p)/p$ is nontrivial. This completes the proof of the Claim. The above computation shows that the global $0$-cycle 
 $z_1=[A,P]-[A,O]-[O,P]+[O,O]\in F^2(X_L)$ has a nontrivial image $(z_{1v}, z_{1\overline{v}})$ under the diagonal map $F^2(X_L)/p\xrightarrow{\Delta}\prod_{v|p}F^2(X_{L_v})_{\nd}$.  
 
A byproduct of the proof of Claim 1 is that the following two maps are  isomorphisms,
 \begin{eqnarray*}
 f:\mathcal{D}^1/\mathcal{D}^2\rightarrow\Z/p, && g:\mathcal{D}^{p-1}/\mathcal{D}^p\rightarrow\Z/p\\
\;\;\;\; x\mapsto\{x,\widehat{P}_v\}_{L_v/L_v}, && \;\;\;\;y\mapsto\{A_v,y\}_{L_v/L_v}. 
 \end{eqnarray*}

 The extensions $L_v$ and $L_{\overline{v}}$ are isomorphic as abstract fields.  Therefore Claim 1 holds true also for the symbol $\{A_{\overline{v}},P_{\overline{v}}\}_{L_{\overline{v}}/L_{\overline{v}}}\in K_2(L_{\overline{v}}, E_{\overline{v}})/p$ corresponding to the $0$-cycle $z_{1\overline{v}}$. 
  Since $P$ is defined over $\Q$, it follows that  $P_{\overline{v}}=P_v$. 
Moreover, the map $f$ being an isomorphism implies that there exists some $a\in\F_p^\times$ such that 
\[f(A_{\overline{v}})=\{A_{\overline{v}},\widehat{P}_v\}_{L_v/L_v}=a\{A_v,\widehat{P}_v\}_{L_{\overline{v}}/L_{\overline{v}}}.\] 
In other words, if we fix an isomorphism 
$
 F^2(X_{L_v})_{\nd}\xrightarrow{\simeq}\Z/p$ sending
 $z_{1v}$ to $1
$,
then the diagonal map $\Delta:F^2(X_L)/p\rightarrow F^2(X_{L_v})_{\nd}\oplus F^2(X_{L_{\overline{v}}})_{\nd}\xrightarrow{\simeq}\Z/p\oplus\Z/p$, $z\mapsto(z_{v},z_{\overline{v}})$ maps $z_1$ to the tuple $(1,a)$. 

\textbf{Claim 2:} The element $\Delta(z_2)$ is $\F_p$-linearly independent from $\Delta(z_1)=(1,a)$, where $z_2$ is the $0$-cycle $z_2=[A,\omega(P)]-[A,O]-[O,\omega(P)]+[O,O]\in F^2(X_L)$. 
 
The $0$-cycles $z_{2v},z_{2\overline{v}}$ correspond to the symbols $\{A_v,\widehat{\omega(P)}_v\}_{L_v/L_v}$, $\{A_{\overline{v}},\widehat{\omega(P)}_{\overline{v}}\}_{L_{\overline{v}}/L_{\overline{v}}}$. Since $P\in E(K)\subset E(L)$, the point $\omega(P)_{\overline{v}}$ is the restriction of $\omega(P)_{\overline{\mathfrak{p}}}$, and hence it is none other than the complex conjugate, $\overline{\omega}(P)_v$. The points $\omega(P), \overline{\omega}(P)$ induce elements $\widehat{\omega(P)}_v, \widehat{\overline{\omega}(P)}_v\in \mathcal{D}^{p-1}/\mathcal{D}^p$. Since 
 $\omega(P)+\overline{\omega}(P)=-P$, and $\widehat{P}_v\in\mathcal{D}^{p-1}/\mathcal{D}^p$ is nontrivial,  at least one of these points induces a nontrivial element of $\mathcal{D}^{p-1}/\mathcal{D}^p$. Without loss of generality, assume  there exists $m\in\F_p^{\times}$ such that $\widehat{\omega(P)}_v\equiv m\widehat{P}_v\mod\mathcal{D}^p$. 
The isomorphism $g$ then yields an equality
\[g(\omega(P)_v)=\{A_v,\omega(P)_v\}_{L_v/L_v}=m\{A_v,P_v\}_{L_v/L_v}.\] To prove linear independence, it is enough to show  \[\{A_{\overline{v}},\overline{\omega}(P)_v\}_{L_{\overline{v}}/L_{\overline{v}}}\neq ma\{A_v,P_v\}_{L_v/L_v}.\] Notice that we have  
 $\{A_{\overline{v}},\overline{\omega}(P)_v\}_{L_v/L_v}=a\{A_v,\overline{\omega}(P)_v\}_{L_v/L_v}.$ Thus, it suffices to prove that the elements $\widehat{\omega(P)}_v,\widehat{\overline{\omega}(P)}_v$ induce distinct elements of $\mathcal{D}^{p-1}/\mathcal{D}^p$. Suppose for contradiction that they are the same. That is, $\widehat{\omega(P)_v}\equiv\widehat{\overline{\omega}}(P)_v\equiv m\widehat{P}_v\mod\mathcal{D}^p$. We have $\widehat{\omega(P)}_v+\widehat{\overline{\omega}}(P)_v=-\widehat{P}_v$, giving $2m=-1$ in $\F_p$. Moreover, $\widehat{\overline{\omega}(\omega(P))}_v=m^2\widehat{P}_v=c\widehat{P}_v$, where $c=(D+1)/4$. That is, $m\in\F_p^\times$ is a solution to the system \[\left\{\begin{array}{c}
x^2=c\\
2x=-1
\end{array}\right.,\] which is a contradiction by \autoref{congruence}. 
 


\end{proof} 

\begin{rem}\label{computations} 
 We believe that the condition of having a ``good" rational point $P\in E(\Q)$ is satisfied by a positive proportion of elliptic curves with positive rank over $\Q$. As we saw in the proof of \autoref{mainmain2}, all we need is a point $P$ such that $\widehat{P}_{\Q_p}\in\widehat{E}^1(\Q_p)\setminus\widehat{E}^2(\Q_p)$. The latter can be checked computationally as follows. 
\begin{enumerate}
\item Compute the coordinates of a point $P\in E(\Q)$ of infinite order. 
\item Compute the coordinates of the induced local point $\widehat{P}_{\Q_p}\in E(\Q_p)$. 
\item Compute the coordinates of a local nonzero $p$-torsion point $P_{0}\in E[p](\Q_p)$.
\item Find a scalar $\lambda\in \F_p$ such that the $x$-coordinate of $P_{\Q_p}-\lambda P_{0}$ has negative valuation. 
\item If the above valuation is exactly $-2$, then $P$ is a point that satisfies the condition of \autoref{mainmain2}.
\end{enumerate} When the ordinary prime $p$ is small enough, this is a computation that can be carried out in SAGE. As a case study, in the appendix \ref{appendix} we carry out this computation for the family of elliptic curves $\{y^2=x^3-2+7n:n\in\Z\}$, and the prime $p=7$. 
\end{rem}

\begin{rem} An interesting case to consider next is the following. Start with a pair $(E,p)$ that satisfies the assumptions of \autoref{mainmain1} with $F_0=\Q$ and consider extensions $M/\Q$ of degree $g<p-1$ such that $p$ splits completely in $M$. Then we still have a vanishing $\Br(X_{L\cdot M})\{p\}/\Br_1(X_{L\cdot M})\{p\}=0$, where $L=K(E[\pi])$. An analog of \autoref{mainmain1} gives 
\[\prod_{v|p}F^2(X_{(LM)_v})_{\nd}\simeq(\Z/p)^{2g}.\]  We see that as $g$ gets larger, it becomes increasingly harder to find concrete criteria that guarantee the lifting of this group to global $0$-cycles. 
For example, consider the curve $E$ given by the Weierstrass equation $y^2=x^3+2$ and the prime $p=61$. Then $p$ splits completely over the degree six extension $\Q(\alpha)$, where $\alpha$ is a root of the polynomial $x^6-x^4+x^2-3x+3$.
\end{rem}
\vspace{2pt}
The key ingredient in the proof of \autoref{mainmain2} was that the two homomorphisms 
 \begin{eqnarray*}
 f:\mathcal{D}^1/\mathcal{D}^2\rightarrow\Z/p, && g:\mathcal{D}^{p-1}/\mathcal{D}^p\rightarrow\Z/p\\
\;\;\;\; x\mapsto\{x,\widehat{P}_v\}_{L_v/L_v}, && \;\;\;\;y\mapsto\{A_v,y\}_{L_v/L_v}
 \end{eqnarray*} are isomorphisms and this was because the residue field of $K(E[\pi])$ is only $\F_p$. A similar method can be extended to many more situations. For example the following corollary describes a generalization to when $L=F_0(E[\pi])$ with $F_0/\Q$ a quadratic extension. 
 \begin{cor}\label{quadraticcase} Let $(E,p)$ be a pair that satisfies the assumptions of \autoref{mainmain1}. Suppose that $p\nmid|\overline{E_p}(\F_p)|$, but $p||\overline{E_p}(\F_{p^2})|$. Let $F_0/\Q$ be a quadratic extension having a unique inert place $w_0$  above $p$. Set $k=F_0\cdot\Q_p$. Suppose there exists a point $P\in E(F_0)\setminus E(\Q)$ of infinite order such that the induced local point $P_k\in E(k)$ has the property $\widehat{P}_k\in\widehat{E_k}(k)/p$ is nontrivial. Then the global $p$-torsion $0$-cycles $z_1,z_2\in F^2(X_L)$ induced by $A, P$ and $A,\omega(P)$ respectively  lift the group $\displaystyle\prod_{v|p}F^2(X_{L_v})_{\nd}\simeq\Z/p\oplus\Z/p$. 
 \end{cor} 
 \begin{proof}
 The proof is essentially the same as in \autoref{mainmain2}. 
 In this case we have an isomoprhism of $\F_p$-vector spaces (cf. \cite[Lemma 2.1.4]{kawachi2002}) 
 \[\overline{U}_{L_v}^{p-1}/\overline{U}_{L_v}^{p}\simeq\mathcal{D}_{L_v}^{p-1}/\mathcal{D}_{L_v}^p\xrightarrow{\simeq}\F_{p^2}\simeq\Z/p\oplus\Z/p.\] 
 It follows by \cite[V.3, Corollary 7]{Serre1979local} that there exists some $B\in \mathcal{D}_{L_v}^{p-1}/\mathcal{D}_{L_v}^p$ such that $\displaystyle(\delta(A_v),\delta(B))_p\neq 0$. We will show that we can take $B=(P_k)_v$. Consider the extension $F=K(E[\pi])$ and denote by $v_0$ the unique place of $F$ below $v$. Since $p\nmid|\overline{E_p}(\F_p)|$, we have a proper inclusion $F\subsetneq L$. It follows by \autoref{minimality} that $K_2(F\cdot\Q_p;E_{F\Q_p})/p=0$. This implies that for every $C\in \mathcal{D}_{L_v}^{p-1}/\mathcal{D}_{L_v}^p$ which is the restriction of some element defined over $F_{v_0}$, it follows $(\delta(A_v),\delta(C))_p=0$. Since $\overline{U}_{L_v}^{p-1}/\overline{U}_{L_v}^{p}$ is two dimensional over $\F_p$, every element $B$ that is not in the image of the restriction map $\res_{L_v/F_{v_0}}$ gives rise to a nontrivial symbol. The assumption $P\in E(F_0)\setminus E(\Q)$ implies that we can take $B=(P_k)_v$, which yields the desired nonvanishing $\displaystyle(\delta(A_v),(P_k)_v)_p\neq 0$.
 
  The rest of the proof of \autoref{mainmain2} carries through to show that the $0$-cycles $z_1=[A,P_k]-[A,O]-[O,P_k]+[O,O]$ and $z_2=[A,\omega(P_k)]-[A,O]-[O,\omega(P_k)]+[O,O]$ are $\F_p$-linearly independent. The only difference is that  we have to consider also the case when $K=\Q(\sqrt{D})$ with $D\not\equiv 1\mod 4$ (see \autoref{Zi} for examples of elliptic curves with potential CM by $\Z[i]$ which satisfy the assumptions of the corollary). The computation of linear independence is in fact simpler in this case. The minimal polynomial of $\omega$ is $\mu(x)=x^2-D$, giving $\omega(P)=-\overline{\omega}(P)$. Since $p>2$, it is clear that the elements $\widehat{\omega(P)}_v,\widehat{\overline{\omega}(P)}_v$ induce distinct elements of $\mathcal{D}^{p-1}/\mathcal{D}^p$.
 
 \end{proof}

\subsection*{Extensions}\label{nogoodpoint}
We close this article by discussing alternative criteria of lifting when a ``good" rational point does not exist. We focus on the simplest case considered in \autoref{mainmain2}, namely when $|\overline{E_p}(\F_p)|=p$. In this case the  point $A\in E[\pi]$ can no longer be used to produce global $0$-cycles. One needs to find instead two ``good" points $P,Q\in E(L)$ such that the induced symbol $\{\widehat{P}_v,\widehat{Q}_v\}_{L_v/L_v}\neq 0$. To look for such points we need to start with points of infinite order living in intermediate extensions $\Q\subsetneq F\subseteq L=K(E[\pi])$. In the following paragraphs we develop this idea.  

 Suppose we have a local point $B_v\in \widehat{E}(L_v)$ such that $B_v\in \widehat{E}^l(L_v)\setminus\widehat{E}^{l+1}(L_v)$ for some integer $0<l<p$. Then it follows by \cite[V.3, Corollaries 2 \& 3]{Serre1979local} that the map 
 \[\mathcal{D}^{p-l}/\mathcal{D}^{p-l+1}\rightarrow\Z/p,\;\;\; x\mapsto\{x,B_v\}_{L_v/L_v}\] is an isomorphism. Thus, if we can find intermediate extensions $K\subset M, F\subset L$ and points $P\in E(M), Q\in E(F)$ of infinite order such that the induced local formal points $\widehat{P}_v, \widehat{Q}_v$ lie in $\widehat{E}^{l}(L_v)\setminus\widehat{E}^{l+1}(L_v),$ and $ \widehat{E}^{p-l}(L_v)\setminus\widehat{E}^{p-l+1}(L_v)$ respectively, then the symbol $\{\widehat{P}_v,\widehat{Q}_v\}_{L_v/L_v}$ is nonvanishing in $ K_2(L_v;E_v)/p$. Then we can produce a second $\F_p$-linearly independent symbol using the complex multiplication, as we did in the proof of \autoref{mainmain2}. We next describe a potential algorithm to compute such ``good points" $P,Q$ in a particular example. A similar method can be used more generally, but the algorithm will require more steps. 




\subsection*{An extended example} To illustrate our method, we consider  the case when the quadratic imaginary field $K=\Q(\zeta_3)$, and $p=7$.
 We consider the family of elliptic curves \[\{E_n:y^2=x^3-2+7n,\;n\in\Z\},\] so that the pair $(7, E_n)$ satisfies the assumptions of \autoref{mainmain1}, for all $n\in\Z$.   Suppose that for a certain value of $n$ the condition of \autoref{mainmain2} is not satisfied. That is, there is no good point $P\in E_n(\Q)$. For simplicity we will write $E=E_n$. 
 
The Galois group $G=\Gal(L/K)$ is cyclic of order $6$, and hence there exist unique intermediate subfields $K\subset M, F\subset L$ such that $[M:K]=\displaystyle 2$ and $\displaystyle[F:K]=3$. 
Moreover, the extension $L/K$ is totally ramified at $\mathfrak{p},\mathfrak{\overline{p}}$. This implies that there exist unique places $s, \overline{s}$ of $M$ lying over $\mathfrak{p}, \mathfrak{\overline{p}}$ respectively and the extension $M_s/\Q_p$ is totally ramified of degree $2$.  Similarly there exist unique places $t, \overline{t}$ of $F$ over $\mathfrak{p}, \mathfrak{\overline{p}}$ with $F_t/\Q_p$ totally ramified of degree $3$.

Suppose we can find global points $B\in E(M)$, $C\in E(F)$ such that $\widehat{B}_{s}\in\widehat{E}^1(M_s)\setminus \widehat{E}^2(M_s)$, and $\widehat{C}_t\in\widehat{E}^1(F_t)\setminus \widehat{E}^2(F_t)$. Their restrictions $\widehat{B}_{v}, \widehat{C}_{v}$ in $\widehat{E}(L_v)$ lie in 
$\widehat{E}^3(L_v)\setminus \widehat{E}^4(L_v)$ and $\widehat{E}^2(L_v)\setminus \widehat{E}^3(L_v)$ respectively. 

Next consider the points $\omega(C), \omega^2(C)\in E(F)$. Similarly to the proof of \autoref{mainmain2}, at least one of these points induces a formal point in $\widehat{E}^2(L_v)\setminus \widehat{E}^3(L_v)$. Call this point $C'$. Then there exists some scalar $m\in\F_7^\times$ such that $\widehat{C}'_v\equiv m\widehat{C}_v\mod \mathcal{D}^3$, This is equivalent to saying that the $x$-coordinate of the formal point $\widehat{C}'_t-m\widehat{C}_t$ has strictly smaller valuation than the one of $\widehat{C}_t$. If it happens that $\widehat{C}'_t-m\widehat{C}_t\in \widehat{E}^2(F_t)\setminus \widehat{E}^3(F_t)$, then the point $Q=C'-mC$ is a ``good" point. That is, $\widehat{Q}_v\in \widehat{E}^4(L_v)\setminus \widehat{E}^5(L_v)$, and since $3+4=7$, our earlier discussion shows that 
$\{B_v,Q_v\}_{L_v/L_v}\neq 0\in K_2(L_v;E)/p$.

\appendix
\section{Computations with local points By Angelos Koutsianas}\label{appendix}

In this appendix we give more details about the computations we have done in order to verify the existence of a ``good'' rational point $P$ that Theorem \ref{mainmain2} requires for the family of elliptic curves 
$$
\{E_n:y^2 = x^3 - 2 + 7n,~n\in\Z\}.
$$
Our computations are based on Remark \ref{computations}. We consider the cases $n\in[-\computationsUpperBound, \computationsUpperBound]$ with $p=7$ and $\rk(E_n(\Q))=1$.

For a fixed value of $n$ we define $a=-2+7n$. We recall that $E_n$ has complex multiplication by the ring of integers $\mathcal{O}_K$ of $K=\Q(\sqrt{-3})$. It follows that $p=7$ splits in $\mathcal{O}_K$ with $7=\pi\overline{\pi}$ where $\pi=\frac{1+3\sqrt{-3}}{2}$ and $\overline{\pi}$ is the complex conjugate of $\pi$. For the computations we need an explicit description of $E_n[p](\Q_p)$.

\begin{lem}\label{lem:torsion_group}
With the above notation, it holds 
$$
E_n[7](\Q_7)=\{(\sqrt[3]{\theta}\zeta_3^i, \pm\sqrt{\theta + a}):~\text{for }i=0,1,2\},
$$ 
where $\theta$ is the root of $f(x) = 7x^2 - 4ax + 16a^2$ such that $\theta\equiv 6\pmod{7}$ and $i=0,1,2$.
\end{lem}

\begin{proof}
A symbolic computation shows that the 7th torsion polynomial of $E_n$ is given by
$$
f_7(x) = (7x^6 - 4ax^3 + 16a^2)(x^{18} + 564ax^{15} - 5808a^2x^{12} - 123136a^3x^9 - 189696a^4x^6 - 49152a^5x^3 + 4096a^6).
$$
It follows by example \ref{ex1} that $|\overline{E_{n,p}}(\F_7)|=7$. Moreover, the leading coefficient of $f_7$ is divisible by $7$ and thus by \cite[Theorem VII.3.4]{Silverman2009} we conclude that $\#E_n[7](\Q_7)= 7$ . Therefore, it is enough to show that the above points have order $7$ and coordinates in $\Q_7$.

Let $g(x)=7x^2 - 4ax + 16a^2$, then we can easily prove that the roots of $g$ are 
$$
\theta_{1,2}=2a\frac{1\pm 3\sqrt{-3}}{7}.
$$
Suppose we embed $\sqrt{-3}$ into $\Q_7$ such that $\sqrt{-3}\equiv 2 + 5\cdot 7 + 6\cdot 7^3 + O(7^4)$. As a result, we get $\theta_{1}=2a(2 + 2\cdot 7 + 4\cdot 7^2 + O(7^3))$ and $\theta_2 = 2a(\frac{2}{7} + 5 + 4\cdot 7 + 2\cdot 7^2 + O(7^3))$. Because $v_7(\theta_2)=-1$ we understand that $\sqrt[3]{\theta_2}\not\in\Q_7$.

Let $\theta=\theta_1$. Because $\theta\equiv 6\pmod{7}$ by Hensel's lemma we get that the polynomial $x^3 - \theta$ has a root in $\Q_7$. Furthermore, because $\zeta_3\in\Q_7$ we conclude that $\sqrt[3]{\theta}\zeta_3^i\in\Q_7$ for all $i=0,1,2$.

Let $x_i=\sqrt[3]{\theta}\zeta_3^i$ for $i=0,1,2$. In order to finish the proof, suffices to show that the polynomial
$$
x^2 - (x_i^3 + a)=x^2 - (\theta + a),
$$
has both of its roots in $\Q_7$. Because $\theta + a\equiv 4\pmod{7}$ Hensel's lemma yields the conclusion. 
\end{proof}

Having determined the set $E_n[7](\Q_7)$, we follow the steps in Remark \ref{computations}. We focus on the curves $E_n$ with rank $1$. The most computationally expensive part is the determination of the generator $P$ of the free part of $E_n(\Q)$. In order to speed up the computations we assume the finiteness of Tate-Shafarevich group. 

\begin{theo}\label{computationsthm}
Let $n\in [-5000, 5000]$. Under the assumption of the finiteness of the Tate-Shafarevich group of $E$, the ~86,68\% of the rank one elliptic curves $E_n$ satisfy the hypothesis of \autoref{mainmain2}, in other words there exists a ``good'' rational point $P$.
\end{theo}

\begin{proof}
We have written a Sage \cite{sagemath} script\footnote{See function \textit{rank\_one\_elliptic\_curves}.} that does all the computations we have described in Remark \ref{computations} and Lemma \ref{lem:torsion_group}.

It is important to mention that there are $\missingCurves$ values of $n$ for which we are not able to compute the generator of the free part of $E_n(\Q)$. For the calculation of the above percent we do not consider these $176$ curves. The total amount of time for the computations was ~31 minutes in a regular personal computer.
\end{proof}

The code can be found in
\begin{center}
\texttt{https://github.com/akoutsianas/local\_global\_0\_cycles}
\end{center}


\bibliographystyle{amsalpha}

\bibliography{bibfile,bibfileApprox}

\end{document}